\documentclass[10pt,twocolumn,twoside]{IEEEtran}

\usepackage{amsmath}
\usepackage{amsfonts}
\usepackage{amssymb}
\usepackage{amsthm}
\usepackage{bm}
\usepackage{graphicx}
\usepackage{subfigure}
\usepackage{float}
\usepackage{epsfig}
\usepackage{epstopdf}
\usepackage{pdfpages}
\usepackage{mathtools}
\usepackage{booktabs}
\usepackage{multirow,bigdelim}

\newtheorem{lemma}{Lemma}
\newtheorem{corollary}{Corollary}

\newtheorem{theo}{Theorem}

\newcommand{\comment}[1]{}  

\begin{document}
%
\title{On optimal coordinated dispatch for heterogeneous storage fleets with partial availability}
%
%
%

\author{David~Angeli,~\IEEEmembership{Fellow, IEEE},
        Zihang~Dong,~and~Goran~Strbac,~\IEEEmembership{Member, IEEE}

\thanks{
The authors are with the Department of Electrical and Electronic Engineering, Imperial College London, London, SW7 2AZ, UK. {\textit{ (e-mail:  d.angeli@imperial.ac.uk;
zihang.dong14@imperial.ac.uk; g.strbac@imperial.ac.uk)}}}%
\thanks{David Angeli is also with the Dipartimento di Ingegneria dell'Informazione, Universit\`a di Firenze, Italy}
}

%
%

\markboth{Journal of \LaTeX\ Class Files,~Vol.~14, No.~8, August~2015}%
{Shell \MakeLowercase{\textit{et al.}}: Bare Demo of IEEEtran.cls for IEEE Journals}
%



\maketitle

\begin{abstract}
This paper addresses the problem of optimal scheduling of an aggregated power profile (during a coordinated discharging or charging operation) by means of a heterogeneous fleet of storage devices subject to availability constraints. Devices have heterogeneous initial levels of energy,  power ratings and efficiency; moreover, the fleet operates without cross-charging of the units. An explicit feedback policy is proposed to compute a feasible schedule whenever one exist and scalable design procedures to achieve maximum time to failure or minimal unserved energy in the case of unfeasible aggregated demand profiles. Finally, a time-domain characterization of the set of feasible demand profiles using aggregate constraints is proposed, suitable for optimization problems where the aggregate population behaviour is of interest. 
\end{abstract}


%
\IEEEpeerreviewmaketitle

\section{Introduction}
%
%
%
%
\IEEEPARstart{P}{ower} networks, in coming years, are likely to significantly rely on distributed storage assets in order to ease up the task of balancing demand and supply, both during normal operation or in case of power outages. This is expected, on one hand, because of increased penetration of renewable technologies and the volatility of supply it entails; on the other, because of widespread decarbonisation of the transport sector (and consequent adoption of electric vehicles). The flexibility afforded by a considerable amount of storage capacity connected permanently or intermittently to the network has a big potential for limiting peak demands and related costs, and delivering balancing services to the grid. Recent research has systematically classified a significant amount of literature dealing with optimal energy management for storage devices \cite{alinia2020online, Ma2011decentralized, Xin2013a}. Specifically, the work \cite{alinia2020online} introduces an adaptive charging algorithm with the objective of peak‐load management. Reference \cite{Ma2011decentralized} proposes a decentralized strategy for numerous identical electric vehicles (EVs) under the non-cooperative games to minimize the charging cost. Cooperative control of network control theory is developed in \cite{Xin2013a} to ensure the satisfaction of both energy balance and fair utilization among dispersed energy storage systems.

Under such future scenarios, one issue becomes fundamental; how to schedule power profiles of a multitude of storage devices while respecting individual rated power and energy constraints and, at the same time, fulfilling a preassigned aggregate power profile for the fleet. This question arises both during a hypothetical ``discharging operation'', viz. when the fleet is acting as a ``service provider'' and compensates for lack of power due to outages or fluctuations in availability of renewable generation, or during the ``charging phase'', when storage devices are recharged to meet individual energy needs and the aggregate profile is designed so as to possibly reduce peaks in demand or to minimize conventional generation costs.

Such questions have recently received considerable attention from the community (see \cite{weitzel} and references therein for a recent survey and reasoned literature classification based on keywords such as energy storage and optimal policy/strategy/management).
For instance, \cite{dynprog} presents an application of dynamic programming in the estimation of the capacity value of storage devices. 
A dynamic model to approximate the power/energy capacity of aggregations of electric vehicles is developed in \cite{Liu2013planning}. The aggregate flexibility provided by a collection of thermostatically controlled loads (TCLs) is characterized by \cite{Hao2014aggregate} using a stochastic battery model. The authors of \cite{Fortenbacher2015optimal} adopt Model Predictive control (MPC) to control multiple battery sets to track an aggregated set-point trajectory while minimizing battery degradation, battery system and network losses. Reference \cite{Dall2017optimal} proposes an algorithmic framework which controls the dispatchable distributed energy resources (DERs) to the power demand request from transmission system at the feeder substation. Demand dispatch for regulation of the power grid is considered in \cite{Chen2016state}, based on randomized local control algorithms for homogeneous load in a mean field control setting. 
More recently, the control architectures for a fleet of diverse DERs using the packetized energy management dispatch paradigm has been developed in \cite{Espinosa2020a}. Notions of controlled invariant sets have been proposed in \cite{appino}, to achieve optimality preserving aggregation of fleets. In \cite{Jang2021large}, invariant sets are also used to ensure safe coordination of systems with both local and global constraints while a population of homogeneous air conditioners tracks a power trajectory. 
\cite{zhuzhang} introduces an optimal coordination policy for fleets seeking to fulfill a preassigned reference signal (subject to penalty costs for unmet demand) and achieving profit maximization taking into account the service and recharging phases. 
Our approach complements the existing literature in two respects: i) it considers fully heterogeneous fleets, allowing different power ratings, different initial and target energy levels, different non-unity efficiency and different availability windows; ii) it provides guaranteed and scalable optimal solutions which may help an aggregator monitor its flexibility provision in unidirectional power transfer operations, viz. neglecting the recovery phase, by computing exact optimal dispatch profiles in real time.

In this context, we adopt and further develop the approach proposed in \cite{ETA-pscc,ETA-tops,EAST,zachary,ETA-tosg}. In \cite{ETA-pscc} an optimal causal dispatch policy is introduced, for heterogeneous storage fleets 
unable to cross-charge, and seeking to deliver a pre-assigned aggregate demand profile, while maximizing future flexibility, viz. the ability of meeting future power requests. Remarkably, the same policy was first introduced in \cite{nash}, in the context of hydro-reservoirs.
Further optimality properties, including the ability to minimize \emph{time-to-failure} or \emph{unserved energy}  were highlighted in subsequent publications \cite{ETA-tops,ETA-tosg}, while an explicit and remarkably effective time-domain characterization of the set of feasible power profile demands is provided in \cite{zachary}, using the notion of \emph{load duration curves}. An in depth compendium of the theory, with new insights and subsequent interpretations, can be found in \cite{phtr}. 
While such theory has been developed taking into account heterogeneous fleets of devices, it neglects the significant possibility of \emph{partial availability}, viz. the fact that different devices may be connected to the grid during differing time intervals over the considered prediction horizon, as would typically be expected of i.e. \emph{electric vehicles}.

The present paper extends the approach of \cite{phtr} to deal with the case of storage devices with different \emph{availability windows} (or, more generally, availability sets), heterogeneous power ratings and initial energy values. Its contribution is manifold: 
\begin{itemize}
    \item On one hand, it shows how to convert the dispatch design problem for a fleet subject to availability constraints (under a no cross-charging assumption) into the dispatch design of an augmented demand signal for a similar fleet \emph{without} availability constraints and with possible complemented initial energy levels; hence, it broadens applicability of the previous results to the current set-up.
    \item it shows by means of an example, that no single causal policy exists in this case and provides a characterization of the set of aggregate power profiles that a heterogeneous fleet of given initial energy, power ratings and availability sets is able to deliver. Counter-examples show why simpler necessary conditions are unable to capture the full complexity of the feasible set of power profiles.
    \item It highlights how to leverage the new proposed policies in order to derive optimal power schedules with minimum unserved energy or maximum time-to-failure.
\end{itemize}    
    The analysis complements preliminary statements submitted in the conference paper \cite{cdcbatteries} in several directions: by providing detailed proofs in support of the statement of the result in Section \ref{first}, by developing a characterization of the set of feasible power profiles, by proposing results on optimal dispatch policies achieving minimum unserved energy or maximum time-to-failure and by considering novel medium and large scale examples where such approaches are illustrated. 

\section{Problem formulation and preliminary results}
\label{first}

\subsection{System description and objectives}
Let $\mathcal{N}$ denote a finite collection of batteries (of cardinality $N$). Each battery $i \in \mathcal{N}$ is constrained by a rated power
$\bar{P}_i$, quantifying its maximum discharging rate.
Our aim is to analyze and design dispatch policies for the fleet over a given bounded time interval $\mathcal{T} \subset [0, + \infty)$.
To this end, let $E_i(t)$ denote the state of charge of the $i$-th device at time $t \in \mathcal{T}$, viz. the amount of externally measured energy stored in the device. Notice that measuring energy externally allows to factor in possibly heterogeneous and non-unity efficiencies, in particular, by defining
$E_i(t) = \eta_i \tilde{E}_i(t)$ where $\tilde{E}_i$ denotes internally measured energy and $\eta_i$ the round-trip efficiency coefficient.
The differential equations given below describe the time-evolution of $E_i(t)$:
\begin{equation}
\label{popequation}
\dot{E}_i (t) = - u_i(t)
\end{equation}
with initial configuration of energy levels, $E_i(0)$, $i \in \mathcal{N}$. The variable $u_i(t)$ is the instantaneous power delivered by the $i$-th battery, which needs to fulfill:
\begin{equation}
 \label{pointwise}
 u_i(t) \in [0, \bar{P}_i], \quad \forall \, t \in \mathcal{T}.
 \end{equation}
Together with constraints (\ref{pointwise}), we consider the additional possibility of devices operating within a pre-assigned availability window  (or more in general availability set) $\mathcal{A}_i \subset \mathcal{T}$, so that:
\begin{equation}
\label{window}
u_i(t) = 0, \quad \forall \, t \in \mathcal{T} \backslash \mathcal{A}_i. 
\end{equation} 
Our first objective is to ascertain, for any given power profile $d: \mathcal{T} \rightarrow [0, + \infty)$,
if there exists a control action $u_i(\cdot)$, $i \in \mathcal{N}$, fulfilling constraints
(\ref{pointwise}) and (\ref{window}), while at the same time
\begin{equation}
\label{totalpower}
\sum_{i \in \mathcal{N}} u_i (t) = d(t), \quad \forall \, t \in \mathcal{T},
\end{equation}
and the associated solution of (\ref{popequation}) fulfills $E_i (t) \geq 0,\, \forall \, t \in \mathcal{T}$.

For ease of notation, we arrange energy values in a vector $E(t):= [ E_1(t), E_2 (t), \ldots, E_N (t) ]'$. Moreover, for any given $E(0)$ we define the set of feasible power profiles:
\begin{equation}\label{feset}
\begin{array}{rcl}
\mathcal{F} ( E(0) ) &:=& \Big \{ d(\cdot): \mathcal{T} \rightarrow [0,+ \infty): \\
& &\exists \{ u_i : \mathcal{T} \rightarrow [0,\bar{P}_i] \}_{i \in \mathcal{N}}: \\
& & u_i(t) = 0, \; \forall \, t \notin \mathcal{A}_i, \forall \, i \in \mathcal{N}\\
& &E_i(0) \geq \int_\mathcal{T} u_i(\tau) d \tau,  \, \forall \, i \in \mathcal{N}  \\
& & d(t) = \sum_{i \in \mathcal{N}} u_i(t), \forall \, t \in \mathcal{T} \Big \}.
\end{array}
\end{equation}
Hence, our preliminary task is to find out if $d(\cdot) \in \mathcal{F} (E(0))$, and if so, what is a suitable dispatch policy.
In practice, this question arises whenever a fleet of storage devices are required to coordinate in delivering energy (without cross-charging) to jointly fulfil a given power reference signal $d(t)$.
Previous literature, \cite{ETA-pscc,ETA-tosg}, has answered such questions for the case of full availability, viz. windows coinciding with $\mathcal{T}$.
We approach the problem by suitably modifying the scheduling proposed 
in \cite{ETA-pscc} (later denoted Greedy Greatest Discharge Duration First policy) to address the issue of availability windows. 

Our main contribution is a constructive design algorithm for a feasible policy and a supporting theoretical analysis, showing that the problem can be equivalently framed as one of delivery of an auxiliary (increased) power profile for a fleet with suitably augmented initial energy levels, identical power ratings and full availability over the considered interval.  

As in \cite{ETA-pscc}, we introduce a new set of coordinates, 
the so called \emph{time-to-discharge} variables, defined as:
\begin{equation}
\label{timetogodiff}
x_i (t) = E_i (t) / \bar{P}_i.
\end{equation}
Accordingly, the state evolution is governed by:
\begin{equation}
 \dot{x}_i (t) = - u_i (t)/ \bar{P}_i, \quad x_i (0) = E_i (0) / \bar{P}_i.
\end{equation}
Some additional notations are useful to formulate the proposed feedback policy. More closely, for any discharge duration
$\tau$, we denote the set of agents $\mathcal{N}_\tau (x) := \{ i \in \mathcal{N}: x_i = \tau \}$, where $x \in \mathbb{R}^N$ is the stacked state vector of time-to-discharge over all devices.
Clearly, $\mathcal{N} = \bigcup_{\tau \geq 0} \mathcal{N}_{\tau} (x)$, and this partitions $\mathcal{N}$ as $\mathcal{N}_{\tau_1} \cap \mathcal{N}_{\tau_2} = \emptyset$ for $\tau_1 \neq \tau_2$.
Overall, only a finite number of $\mathcal{N}_{\tau}$ are non-empty, at each time $t$, and we order the corresponding discharge time as, $\tau_1 (t)> \tau_2(t) > \tau_3(t) > \ldots > \tau_{G(t)} (t)$, with $G(t) \leq N$. 

\subsection{Greedy-Greatest-Discharge-Duration First policy}

It is useful to first introduce the Greedy-Greatest-Discharge-Duration First (GGDDF) policy without any reference to availability sets:
\begin{equation}
\label{feedback2}
u_i(t) = \left \{ \begin{array}{rl} \bar{P}_i & \textrm{if } i \in \mathcal{N}_{\tau_k} \textrm{ and }\\& \sum_{h \leq k} \sum_{j \in \mathcal{N}_{\tau_h}} \bar{P}_j
\leq d(t) \\
\tilde{r}(t) \bar{P}_i & \textrm{if }  i \in \mathcal{N}_{\tau_k} \textrm{ and }\\& \sum_{h < k} \sum_{j \in \mathcal{N}_{\tau_h}} \bar{P}_j
\leq d(t)\\& d(t) <  \sum_{h \leq k} \sum_{j \in \mathcal{N}_{\tau_h} } \bar{P}_j \\
0 & \textrm{otherwise} \end{array} \right .
\end{equation}
where the value of $\tilde{r}(t) \in [0,1]$ is determined according to:
\[    \tilde{r}(t) = \frac{ d(t) -  \sum_{h < k} \sum_{j \in \mathcal{N}_{\tau_h} \bar{P}_j }}{ \sum_{j \in \mathcal{N}_{\tau_k} } \bar{P}_j } \]
where $k$ is such that 
\[ \sum_{h < k} \sum_{j \in \mathcal{N}_{\tau_h}} \bar{P}_j
\leq d(t) <  \sum_{h \leq k} \sum_{j \in \mathcal{N}_{\tau_h}} \bar{P}_j. \]
Accordingly, the amount of power extracted from all devices (regardless of availability) equals the instantaneous demand $d(t)$, viz.
\begin{equation}\label{poweraggregation2}
\sum_{i \in \mathcal{N}} u_i (t) = d(t),
\end{equation}
For later use, we denote this feedback policy explicitly as:
\begin{equation}
\label{policy2}
u(t) = \tilde{K} (x(t),d(t)),
\end{equation}
with the associated system of equations:
\begin{equation}
\label{feedbackequation2}
\dot{x} (t) = - \Omega \tilde{K} (x(t),d(t) ).
\end{equation}
where $\Omega = \textrm{diag} [\bar{P}_1^{-1}, \bar{P}_2^{-1},\ldots, \bar{P}_N^{-1} ]$. The solution of (\ref{feedbackequation2}) is denoted as $\tilde{\varphi} (t, x, d(\cdot))$.

While taking into account the availability sets, the explicit feedback policy to dispatch $d(t)$ at time $t$ is expressed as
\begin{equation}
\label{feedback}
u_i(t) = \left \{ \begin{array}{rl} \bar{P}_i & \textrm{if } i \in \mathcal{N}_{\tau_k} \textrm{ and } \\ & \sum_{h \leq k} \sum_{j \in \mathcal{N}_{\tau_h}: t \in \mathcal{A}_j} \bar{P}_j
\leq d(t) \\
r(t) \bar{P}_i & \textrm{if }  i \in \mathcal{N}_{\tau_k} \textrm{ and}  \\
& \sum_{h < k} \sum_{j \in \mathcal{N}_{\tau_h}: t \in \mathcal{A}_j} \bar{P}_j
\leq d(t) \\ & d(t) <  \sum_{h \leq k} \sum_{j \in \mathcal{N}_{\tau_h}: t \in \mathcal{A}_j} \bar{P}_j \\
0 & \textrm{otherwise} \end{array} \right .
\end{equation}
where the value of $r(t) \in [0,1]$ is determined according to:
\[    r(t) = \frac{ d(t) -  \sum_{h < k} \sum_{j \in \mathcal{N}_{\tau_h}: t \in \mathcal{A}_j} \bar{P}_j }{ \sum_{j \in \mathcal{N}_{\tau_k}: t \in \mathcal{A}_j } \bar{P}_j } \]
where $k$ is such that 
\[ \sum_{h < k} \sum_{j \in \mathcal{N}_{\tau_h}: t \in \mathcal{A}_j} \bar{P}_j
\leq d(t) <  \sum_{h \leq k} \sum_{j \in \mathcal{N}_{\tau_h}: t \in \mathcal{A}_j} \bar{P}_j. \]	
In this case, it is remarked that:
\begin{equation}\label{poweraggregation}
\sum_{i \in \mathcal{N} : t \in \mathcal{A}_i} u_i (t) = d(t),
\end{equation}
viz. the amount of power extracted from the available devices equals the instantaneous demand $d(t)$. Moreover, this feedback policy is explicitly denoted as:
\begin{equation}
\label{policy}
u(t) = K (t, x(t),d(t)),
\end{equation}
which is referred as the GGDDF policy with availability sets.
We remark that time-dependence is made explicit in (\ref{policy}) since the availability of devices at any time $t$ directly influence the amount of power extracted from each battery.
With this in mind, the fleet's closed-loop equations read:
\begin{equation}
\label{feedbackequation}
\dot{x} (t) = - \Omega K (t,x(t),d(t) ).
\end{equation}
Of course, the schedule (\ref{policy}) might induce solutions
of (\ref{feedbackequation}) violating the constraint that $x_i(t) \geq 0$ due to insufficient initial energy; we ignore this, for the time being, as our main goal is to compute how much energy the feedback policy extracts outside of the availability window for the given power profile $d(t)$.
Notice also that the constraint $u_i(t)=0$ for $t  \notin \mathcal{A}_i$ is, for the time-being, not enforced. Specifically if a device has a sufficiently high time-to-discharge, it will be discharged by (\ref{policy}) regardless of its availability window.
Denote by
$\varphi(t, x(0), d(\cdot))$ the solution of  (\ref{feedbackequation}) at time $t$ with initial condition $x(0)$.
Following similar steps as in \cite{michaelphd} one can show the properties stated below:
\begin{enumerate}
\item For all initial conditions $x(0) \in \mathbb{R}^N$ there exists a unique Filippov solution (see \cite{Filippov}) of \eqref{feedbackequation} defined for $t \geq 0$.
\item The relative ordering of time-to-discharge is preserved along solutions. Namely, for all $t \geq \tau \geq 0$, it holds
\[ x_i(\tau) \geq x_j(\tau) \Rightarrow \varphi_i (t, x(0), d(\cdot)) \geq \varphi_j (t, x(0), d(\cdot)) \]
\item A corollary of the previous implication is:
	\[ x_i(\tau) = x_j(\tau) \Rightarrow \varphi_i (t, x(0), d(\cdot)) = \varphi_j (t, x(0), d(\cdot))\]
\item Furthermore, sets of agents with the same discharge duration are monotonically non-decreasing in time (and in fact strictly increasing whenever two or more sets merge with each other as the corresponding discharge time equalize)
\end{enumerate}	
Additional properties will be discussed later as they are instrumental to proving existence of a feasible dispatch policy resulting from (\ref{feedback}) subject to a suitable choice of $x(0)$.

\subsection{Equivalent feasible power dispatching}
Let us introduce a vector 
$\lambda = [ \lambda_i ]_{i \in \mathcal{N}} \in [0,1]^N$ which we use
to define the auxiliary time-to-discharge vector $\tilde{x}_i = x_i + \lambda_i \mu ( \mathcal{T} \backslash \mathcal{A}_i )$, where $\mu$ denotes the Lebesgue measure in $\mathbb{R}$.
Informally $\lambda_i$ modulates between a minimum value of $0$ (no energy) and $1$ (energy corresponding to discharging at full rate)
the auxiliary energy needed by agent $i$ to account for discharging happening outside of its availability window.
For each $\tilde{x}(0)$ we define the corresponding solution according to our feedback policy, viz. $\tilde{x} (t) := \varphi (t, \tilde{x} (0) , d(\cdot)).$
Then, for each battery $i \in \mathcal{N}_i$, we integrate the amount of energy delivered outside of its availability set:
\begin{equation}
\label{outsideenergy}
 \Delta_i = \int_{\mathcal{T} \backslash \mathcal{A}_i} K_i (t, \varphi(t, \tilde{x} (0), d(\cdot)), d(t) ) \, d t  
\end{equation}  
Notice that $\Delta_i$ is a function of $\lambda$, as $\tilde{x}(0)$ is such. In particular, we focus on the map:
\begin{equation}
\label{Lambdamap}
  \Lambda( \lambda ) := [ \Delta_i / (\bar{P}_i \cdot \mu(\mathcal{T} \backslash \mathcal{A}_i) )]_{i \in \mathcal{N}}. 
\end{equation}
The following properties of $\varphi$ and $\Lambda$ were stated in
\cite{cdcbatteries} without proof:
\begin{enumerate}
\item $\Lambda ( [0,1]^N ) \subset [0,1]^N$
\item $\Lambda: [0,1]^N \rightarrow [0,1]^N$ is a continuous function
\item $\varphi(t,x(0), d(\cdot))$ is a cooperative (monotone) system \cite{monotoneangeli}: $x_1 \succeq x_2 \Rightarrow \varphi (t,x_1, d(\cdot) ) \succeq \varphi(t,x_2, d(\cdot)  ), \; \forall \, t \geq 0$, where $x_1$ and $x_{2}$ are temporarily used to denote two different vectors of time-to-discharge, $\succeq$ component-wise inequalities between vectors.
\item $\varphi$ fulfils translation invariance (see \cite{translationinvariance}), viz.
$ \varphi(t,x(0) + \delta \textbf{1}, d(\cdot ) ) = \varphi(t,x(0), d(\cdot) ) + \delta \textbf{1}$, where $\delta \in \mathbb{R}$ is arbitrary and $\textbf{1}$ denotes the vector of all ones of dimension $N$.
\item The flow is weakly contracting with respect to the infinity norm:
$ \| \varphi(t,x_1,d(\cdot)) - \varphi(t,x_2,d(\cdot)) \|_{\infty}
\leq \| x_1 - x_2 \|_{\infty}$.
\end{enumerate}

\begin{proof}
i) \textit{Monotonicity of $\varphi$}: Notice that, the feedback $K$ is such that
$K_i (t,x,d)$ is non-increasing with respect to $x_j$ for all $j \neq i$
(increasing $x_j$ might trigger an increase in its priority so that more energy will be taken from $x_j$ and, consequently, possibly less will be extracted from $x_i$).
Moreover, while $K$ is discontinuous (in fact piece-wise constant in $x$), similar steps as in Proof of Lemma 3.2.1 in \cite{michaelphd}, show that there exists  a unique (Filippov's or Caratheodory) solution that fulfils $\dot{x}(t) = - \Omega K (t,x(t),d(t))$ for almost all $t$.
By uniqueness of solutions, combined with non-decreasingness of $\dot{x}_i$ with respect to $x_j$ for all $i \neq j$, we conclude that system (\ref{feedbackequation})
is cooperative (\cite{monotoneangeli}). Hence, denoting by $\succeq$ componentwise inequalities between vectors, we see that 
$x_1 \succeq x_2 \Rightarrow \varphi(t,x_1,d(\cdot)) \succeq \varphi(t,x_2,d(\cdot)),\, \forall \, t \geq 0$.

ii) \textit{Translation invariance}: Let $\textbf{1}$ denote the vector of all $1$s of dimension $N$.
For all $x \in \mathbb{R}^N$, all $d \geq 0$ and any $\delta \in \mathbb{R}$ it holds $K (t, x+ \delta \textbf{1}, d ) = K (t,x,d)$.
This follows because the relative ordering of time-to-discharge is unaffected by a simultaneous $\delta$ increase (or decrease) affecting all batteries.
Hence, considering any solution $\varphi(t,x,d(\cdot))$ we see that:
$ \frac{d}{dt}  \varphi(t,x,d(\cdot)) + \delta \textbf{1} = - \Omega K (t,\varphi(t,x,d(\cdot)), d(t)) = - \Omega K (t,\varphi(t,x,d(\cdot))+ \delta \textbf{1} , d(t))$.
This proves that 
$ \varphi(t,x,d(\cdot)) + \delta \textbf{1}$ is a solution of system (\ref{feedbackequation}), with initial condition $x+ \delta \textbf{1}$. In other words,
$\varphi(t,x+ \delta \textbf{1}, d(\cdot) ) = \varphi(t,x,d(\cdot)) + \delta \textbf{1}$.

iii) \textit{Weak contraction}: To prove weak contraction, let $x_1, x_2$ be arbitrary in $\mathbb{R}^N$. Let the vectors $\underline{x}$ and $\bar{x} \in \mathbb{R}^N$ be defined as 
$\underline{x} := \min \{ x_1, x_2 \}$ and $\bar{x}:= \max \{x_1, x_2 \}$, where $\min$ and $\max$ are meant component-wise.
 Clearly $\bar{x} \succeq x_1 \succeq \underline{x}$ and $\bar{x} \succeq x_2 \succeq \underline{x}$.
 Exploiting monotonicity we see that, for any $t \geq 0$ and any $d(\cdot)$,
 $\varphi(t,\bar{x} , d(\cdot) ) \succeq \varphi(t,x_i,d(\cdot)) \succeq \varphi(t, \underline{x},d(\cdot) ),\, i=1,2$.
Rearranging the previous inequalities we can show that:
\begin{equation*}
\begin{aligned}
\varphi(t,\bar{x},d) - \varphi(t, \underline{x},d) &\succeq \varphi(t,x_1,d) - \varphi(t,x_2, d)\\& \succeq  -\varphi(t,\bar{x},d) +\varphi(t, \underline{x},d).
\end{aligned}
\end{equation*}
Denoting component-wise absolute values of a vector as $| \cdot |$ the previous inequality can equivalently be written as:
\begin{equation}
\label{combination}
  |\varphi(t,x_1,d) - \varphi(t,x_2, d)| \preceq   \varphi(t,\bar{x},d) - \varphi(t, \underline{x},d).
 \end{equation} 
Define next $\delta:= \max_{i \in \mathcal{N}} |x_{1i} - x_{2i}|$. Clearly, $\bar{x} \preceq \underline{x} + \delta \textbf{1}$, hence, by monotonicity and translation invariance:
\begin{equation}
\label{finalcomb}
 \varphi(t,\bar{x},d) - \varphi(t, \underline{x},d) \preceq
 \varphi(t,\underline{x}+ \delta \textbf{1},d) - \varphi(t, \underline{x},d) 
 = \delta \textbf{1}. 
 \end{equation}
Combining (\ref{combination}) with (\ref{finalcomb}) yields $|\varphi(t,x_1,d) - \varphi(t,x_2, d)| \preceq \delta \textbf{1}$,
and ultimately
$\max_{i \in \mathcal{N}} |\varphi_i(t,x_1,d) - \varphi_i(t,x_2, d)| \leq \delta$.
Denoting by $\| \cdot \|_{\infty}$ infinity norms, the previous inequality reads
$ \| \varphi(t, x_1 , d) - \varphi(t,x_2,d) \|_{\infty} \leq \| x_1 - x_2 \|_{\infty}$, 
which proves weak contractivity of $\varphi$.

iv) \textit{Continuity of $\Lambda$}: We prove continuity of $\Lambda$ under the assumption that each
availability set $\mathcal{A}_j$ is at most the union of a \emph{finite} number of (disjoint) intervals. Because of this, the same is true of the complement:
\begin{equation}
\label{assumavailability}
  \mathcal{T} \backslash \mathcal{A}_j = [t_1^j,t_2^j ] \cup [t_3^j,t_4^j] \cup \ldots \cup [t_{M-1}^j, t_M^j ]  
 \end{equation}
for some even integer $M$.
Let $\mu=[\mu(\mathcal{T} \backslash A_j )  ]_{j \in \mathcal{N}}$.
Then, for each $\lambda$, we define $\tilde{x} (0) = \textrm{diag}(\mu) \lambda$,
and define the associated map as:
$\Lambda_j( \lambda ) =  \frac{1}{\bar{P_j} } \int_{\mathcal{T} \backslash \mathcal{A}_j}
K_j ( t,\varphi(t,\tilde{x}(0),d(\cdot)), d(t) ) \, dt$.
Recalling that $K_j/ \bar{P}_j$ is the derivative of $- x_j$ with respect to time, and exploiting (\ref{assumavailability}) we see that:
\begin{equation*}
    \begin{aligned}
\Lambda_j ( \lambda) = &\varphi_j (t_1^j, \tilde{x}(0),d(\cdot)) - \varphi_j(t_2^j,\tilde{x}(0),d(\cdot))\\& + \ldots + \varphi_j(t_{M-1}^j , \tilde{x}(0),d(\cdot)) - \varphi_j(t_M^j,\tilde{x}(0),d(\cdot)).
\end{aligned}
\end{equation*} 
In order to assess the variation of $\Lambda$ with respect to $\lambda_1$ and $\lambda_2$, we consider, $\tilde{x}_1 (0) = \textrm{diag}(\mu) \lambda_1$ and $\tilde{x}_2 (0) = \textrm{diag}(\mu) \lambda_2$.
By the previous equation, then:
\begin{equation*}
\begin{aligned}
\Lambda_j ( &\lambda_1) - \Lambda_j (\lambda_2)\\ =& \varphi_j (t_1^j, \tilde{x}_1(0),d(\cdot)) - \varphi_j(t_2^j,\tilde{x}_1(0),d(\cdot))\\& + \ldots + \varphi_j(t_{M-1}^j , \tilde{x}_1(0),d(\cdot)) - \varphi_j(t_M^j,\tilde{x}_1(0),d(\cdot))  \\
& - \varphi_j (t_1^j, \tilde{x}_2(0),d(\cdot)) + \varphi_j(t_2^j,\tilde{x}_2(0),d(\cdot)) \\&+ \ldots - \varphi_j(t_{M-1}^j   , \tilde{x}_2(0),d(\cdot)) + \varphi_j(t_M^j,\tilde{x}_2(0),d(\cdot)) \\
=& \sum_{m=1}^M (-1)^{m-1} [ \varphi_j(t_m,\tilde{x}_1(0),d(\cdot)) -
\varphi_j(t_m, \tilde{x}_2(0),d(\cdot)) ] \\ \leq& M \| \tilde{x}_1(0) - \tilde{x}_2(0) \|_{\infty} = \| \textrm{diag} (\mu) [\lambda_1 - \lambda_2] \|_{\infty}. 
\end{aligned}
\end{equation*}
Notice that a similar inequality holds for $\Lambda_j(\lambda_2) - \Lambda_j ( \lambda_1)$. Hence, continuity of $\Lambda$ follows, since $\| \Lambda( \lambda_1) - \Lambda ( \lambda_2 ) \|_{\infty} \leq \bar{M} \max_j ( \mu_j )
 \| \lambda_1 - \lambda_2 \|_{\infty}$, where $\bar{M}$ is the maximum of $M$ over the fleet of devices.

\end{proof}

Notice that, thanks to property $1.$, $2.$ and by virtue of Brouwer's fixed point Theorem, $\Lambda$ admits (at least) one fixed point, namely a value $\bar{\lambda} \in [0,1]^N$ such that $\Lambda ( \bar{\lambda}) = \bar{\lambda}$.
Finally, we recall the main statement in \cite{cdcbatteries}, and later provide details of its proof.
\begin{theo}
	\label{main}
	Consider a fleet of devices $\mathcal{N}$ with initial time-to-discharge, power ratings and availability windows $x_j(0)$, $\bar{P}_j$, and $\mathcal{A}_j$ respectively, for $j \in \mathcal{N}$.  Let $d(\cdot):
	 \mathcal{T} \rightarrow [0, + \infty)$ be the aggregated power demand signal.
	 Define, for convenience, the corresponding energy vector $E_j (0) = x_j(0) \bar{P}_j$. Let $\bar{\lambda}$ be a fixed point of the associated map $\Lambda$
	 and $\tilde{x} (0)$ be the corresponding vector of time-to-discharge, viz.
	 $\tilde{x}_j (0) = x_j(0) + \bar{ \lambda}_j \mu( \mathcal{T} \backslash \mathcal{A}_j)$
	 The following facts are equivalent:
\begin{enumerate}
	\item The signal $d(\cdot)$ belongs to $\mathcal{F} ( E(0) )$ (viz. it is feasible for the fleet $\mathcal{N}$ subject to availability constraints);
	\item The signal:
	 \begin{equation}
	 \label{auxdemand}
	 \tilde{d} (t) = \sum_{j \in \mathcal{N} } K_j (t,\varphi(t,\tilde{x}(0), d(\cdot) ), d(t))
	 \end{equation}
	  is feasible for the fleet $\mathcal{N}$ \emph{without availability constraints} from the initial condition $\tilde{x} (0)$.

\end{enumerate}
	 
\end{theo}
Recalling that the policy $\tilde{K}$ in \eqref{feedbackequation2} differs from the feedback policy $K$ in \eqref{feedbackequation} by the power aggregation equations \eqref{poweraggregation2} and \eqref{poweraggregation}. In addition, the feedback policy $\tilde{K}$ is time-invariant (unlike $K$) as its formulation neglects availability windows. 
In order to prove our main result, Theorem \ref{main}, it is useful to establish a mathematical relation between policies $\tilde{K}$ and $K$ introduced in (\ref{policy2}) and (\ref{policy}), respectively.
\begin{lemma}\label{keylemma}
	For any $x \in \mathbb{R}^N$, any $d \geq 0$ and any $t \in \mathcal{T}$, it holds $K ( t, x, d ) = \tilde{K} \left ( x,d + \sum_{j: t \notin \mathcal{A}_j } K_j (t,x,d) \right )$.
\end{lemma}	
For the sake of readability, we defer the proof of the Lemma to Appendix \ref{bappendix}.
A useful consequence of Lemma \ref{keylemma} is the following alternative expression for solutions of (\ref{feedbackequation2}).
\begin{lemma}
	\label{integratedlemma}
	Let $x \in \mathbb{R}^N$ be arbitrary, and $d:\mathcal{T} \rightarrow \mathbb{R}_{\geq 0}$ denote a given power profile.  Let $\tilde{d} (t) = d(t) + \sum_{j: t \notin \mathcal{A}_j} K_j
	(t,\varphi(t,x,d(\cdot)),d(t) )$. Then, for any $t \in \mathcal{T}$, it holds $\varphi(t,x,d(\cdot)) = \tilde{ \varphi } \left (t,x, \tilde{d} ( \cdot) \right )$.
\end{lemma}
\begin{proof} 
Let $x(t):=  \varphi  (t,x, d(\cdot) )$. Clearly $x(0) = x$, moreover, taking derivatives with respect to time yields:
\begin{equation*}
    \begin{aligned}
\dot{x} (t) =& - \Omega K (t,  x(t), d(t))\\
 = & - \Omega \tilde{K}\left ( x(t), d(t)+\sum_{j: t \notin \mathcal{A}_j} K_j (t,x(t),d(t)) \right ).
     \end{aligned}
\end{equation*}
 Hence $x(t)$ is also solution of (\ref{feedbackequation2}) with initial condition $x$ and input signal $\tilde{d}(t)$. This proves Lemma \ref{integratedlemma}.
 \end{proof}

We are now ready to prove Theorem \ref{main}.
\begin{proof}
We show first the implication $1 \Rightarrow 2$.
Let $d: \mathcal{T} \rightarrow \mathbb{R}_{\geq 0}$ be a feasible power profile with respect to availability sets $\mathcal{A}_j$ ($j \in \mathcal{N}$) and with initial time-to-discharge distribution $x(0)$. Let $\bar{\lambda}$ be a fixed point of $\Lambda(\cdot)$ and $\tilde{x} (0)$ defined according to $\tilde{x}_j (0)
= x_j(0) + \lambda_j \mu( \mathcal{T} \backslash \mathcal{A}_j)$.
Then, there exist $u_j(\cdot): \mathcal{T} \rightarrow [0, \bar{P}_j]$, such that: 
\begin{enumerate}
\item $\sum_{j \in \mathcal{N}} u_j (t) = d(t)$;
\item $u_j(t) =0$ for all $j$ and all $t \notin \mathcal{A}_j$;
\item $\int_{\mathcal{T}} u_j(t) dt \leq x_j (0) \bar{P}_j$.
\end{enumerate}
Consider the following auxiliary input signals:
\[   \tilde{u}_j (t) = \left \{ \begin{array}{rl} u_j(t) & \textrm{if } t \in \mathcal{A}_j \\ K_j ( t, \varphi(t,\tilde{x}(0),d(\cdot)), d(t) ) & \textrm{if } t \notin \mathcal{A}_j. \end{array} \right. \] 
We claim that $\tilde{u}_j$ are a feasible input for demand profile $\tilde{d}$ (without availability restrictions) and for initial time-to-discharge distribution $\tilde{x}(0)$.
To this end notice that:
\begin{equation*}
    \begin{aligned}
     \sum_{j \in \mathcal{N} } \tilde{u}_j(t) =&
	\sum_{j: t \in \mathcal{A}_j } u_j(t) + \sum_{j: t \notin \mathcal{A}_j }
	K_j ( t, \varphi(t,\tilde{x}(0),d(\cdot)), d(t) )	\\
 =&\,\, d(t) +  \sum_{j: t \notin \mathcal{A}_j }
	K_j ( t, \varphi(t,\tilde{x}(0),d(\cdot)), d(t) ) \\
	=& \sum_{j: t \in \mathcal{A}_j }
		K_j ( t, \varphi(t,\tilde{x}(0),d(\cdot)), d(t) ) \\&\,\,+  \sum_{j: t \notin \mathcal{A}_j }
			K_j ( t, \varphi(t,\tilde{x}(0),d(\cdot)), d(t) )= \tilde{d} (t).
\end{aligned}
\end{equation*}
Moreover:
\begin{equation*}
\begin{aligned}
\int_{\mathcal{T}} \tilde{u}_j (t) dt =&
\int_{\mathcal{A}_j}\hspace{-3mm} u_j(t) dt + \int_{\mathcal{T} \backslash \mathcal{A}_j} \hspace{-3mm} K_j ( t, \varphi(t,\tilde{x}(0),d(\cdot)), d(t) ) dt \\
\leq& x_j (0) \bar{P}_j + \Lambda_j(\bar{\lambda}) \mu( \mathcal{T} \backslash \mathcal{A}_j ) \bar{P}_j = \tilde{x}_j (0) \bar{P}_j,
\end{aligned}
\end{equation*}
where the last inequality follows by recalling that $\bar{\lambda}$ is a fixed point of $\Lambda$ and definition of $\tilde{x}(0)$. This completes the proof of our claim. 

We show next the implication $2 \Rightarrow 1$.
Let us assume $\tilde{d}(\cdot)$ be feasible for initial condition $\tilde{x}(0)$ and disregarding availability windows. Then, by the main result in \cite{ETA-pscc},  it can be dispatched through the GGDDF policy, and in particular:
$\tilde{\varphi} (t, \tilde{x} (0), \tilde{d} ) \succeq 0$ for all $t \in \mathcal{T}$.
By Lemma \ref{integratedlemma}, it follows that $\varphi(t,\tilde{x}(0), d(t)) =  \tilde{\varphi} (t, \tilde{x} (0), \tilde{d} ) \succeq 0$ for all $t \in \mathcal{T}$.
Equivalently, for all $j \in \mathcal{N}$:
\begin{equation}
\label{feasconsequence}
\tilde{x}_j(0) - \int_{\mathcal{T}} K_j (t,\varphi(t,\tilde{x}(0),d(\cdot)),d(t)) \, dt   \geq 0 
\end{equation}
We claim that the dispatch policy $u_j(t)$ defined below, proves feasibility of $d(t)$ for the fleet with initial condition $x(0)$ and availability windows $\mathcal{A}_j$:
\begin{equation}
    \label{tobediscussed}
 u_j(t)  = \left \{   \begin{array}{rl} K_j (t,\varphi(t,\tilde{x} (0), d(\cdot)),d(t)) & \textrm{if }
t \in \mathcal{A}_j \\
0 & \textrm{if } t \notin \mathcal{A}_j.
\end{array} \right .
 \end{equation}
 Indeed, $u_j(t) \in [0, \bar{P}_j ]$ for all $t \in \mathcal{T}$. Moreover,
 $u_j(t) = 0$ for $t \notin \mathcal{A}_j$.
 In addition:
\begin{equation*}
\begin{aligned}
 \int_{\mathcal{T}} u_j(t) dt =& \int_{\mathcal{A}_j} K_j(t,\varphi(t,\tilde{x}(0),d(\cdot),d(t)) dt \\
=&
 \int_{\mathcal{T}} K_j(t,\varphi(t,\tilde{x}(0),d(\cdot),d(t)) \, dt
\\ & \,\,-  \int_{\mathcal{T} \backslash \mathcal{A}_j} K_j(t,\varphi(t,\tilde{x}(0),d(\cdot),d(t)) \, dt \\
\leq& \bar{P}_j \tilde{x}(0) - \bar{\lambda}_j \mu(\mathcal{T} \backslash \mathcal{A}_j) \bar{P}_j = \bar{P}_j x(0).
\end{aligned}
 \end{equation*}
where the last inequality follows by (\ref{feasconsequence}) and definition of
$\Lambda (\cdot)$ and $\bar{\lambda}$. This completes the proof of the implication.
\end{proof}

\section{Discussion and Interpretations}

\subsection{Existence of the fixed point $\bar{\lambda}$}
Our first result, Theorem \ref{main}, provides a characterization of feasible aggregate demand profiles for storage fleets in the presence of availability constraints, by converting it to the simpler (and previously addressed) problem of dispatch for a fleet without availability constraints.

Specifically, running a numerical or analytical simulation of the GGDDF policy with availability windows within the time-interval $\mathcal{T}$ allows to compute $\varphi(t,x_0,d(\cdot))$ and to integrate the amount of energy delivered by each agents outside their availability window, according to (\ref{outsideenergy}).
This, for any $\lambda$-dependent initialization $\tilde{x}(0)$, defines the map $\Lambda$, as specified in (\ref{Lambdamap}).
A fixed-point $\bar{\lambda}$ of $\Lambda$ can then be found using a numerical iterative scheme.
Notice that, by continuity of $\Lambda$, and forward invariance of $[0,1]^N$, we may apply Brouwer's fixed point Theorem, \cite{topology}, to conclude that such $\bar{\lambda}$ always exists. 

Extensive simulations have shown that the limit
$ \lim_{k \rightarrow + \infty} \Lambda^k(\tilde{x}(0)) $
exists for any choice of $\tilde{x}(0)$ (though this was not proved formally) leading to conjecture that a fixed point can simply be computed by iterating the map $\Lambda$. As shown in (\ref{tobediscussed}), the GGDDF policy initialized with $\tilde{x}(0)$ computed from $\bar{\lambda}$ allows to use $K_j (t, \varphi(t,\tilde{x}(0),d(\cdot)), d(t)$ restricted to the availability set $\mathcal{A}_j$ as a feasible dispatch input for device $j$. 

This dispatch policy matches a GGDDF policy without availability windows for a suitably inflated demand signal, but unlike the case of full availability, it is not causal, as it requires prior computation of the fixed point $\bar{\lambda}$ which, implicitly, takes into account demand over the whole horizon. 
As shown below, this is not a weakness of the approach, but rather a consequence of the considered set-up. Indeed, considering a two batteries fleet, is enough to show that no causal policy exists in general.

\subsection{Impossibility of causal dispatch}
Consider a fleet $\mathcal{N} = \{1,2\}$, with rated powers 
$\bar{P}_1 = \bar{P}_2 = 1$ and assume an initial discharge time $x_1(0)=3$ and $x_2(0)=6$, respectively.
The availability windows $\mathcal{A}_1 = [0,5]$ and $\mathcal{A}_2=[0,12]=\mathcal{T}$ are assigned.
We propose next two feasible aggregated demand signals, $d_1(t)$ and $d_2(t)$, as shown in Fig.
\ref{feaspow}.
\begin{figure}
\centering
\includegraphics[width=0.23\textwidth]{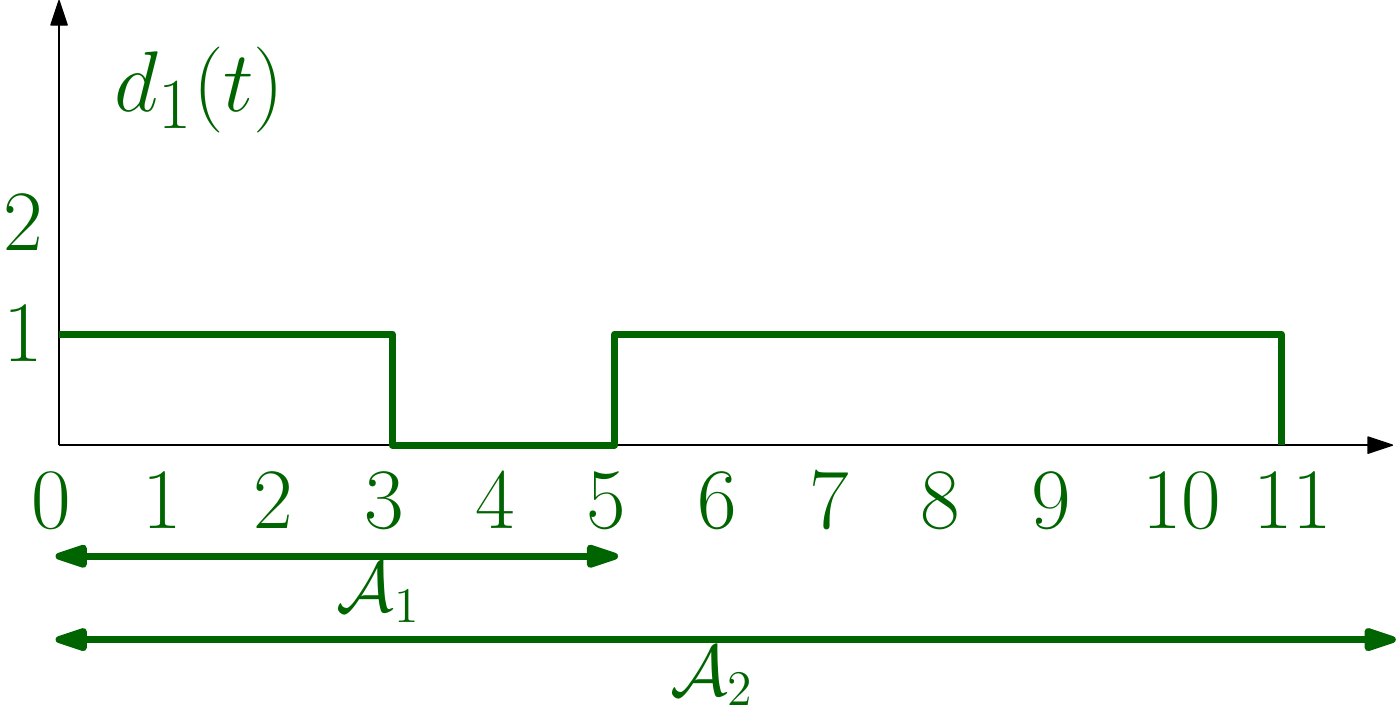}
\includegraphics[width=0.23\textwidth]{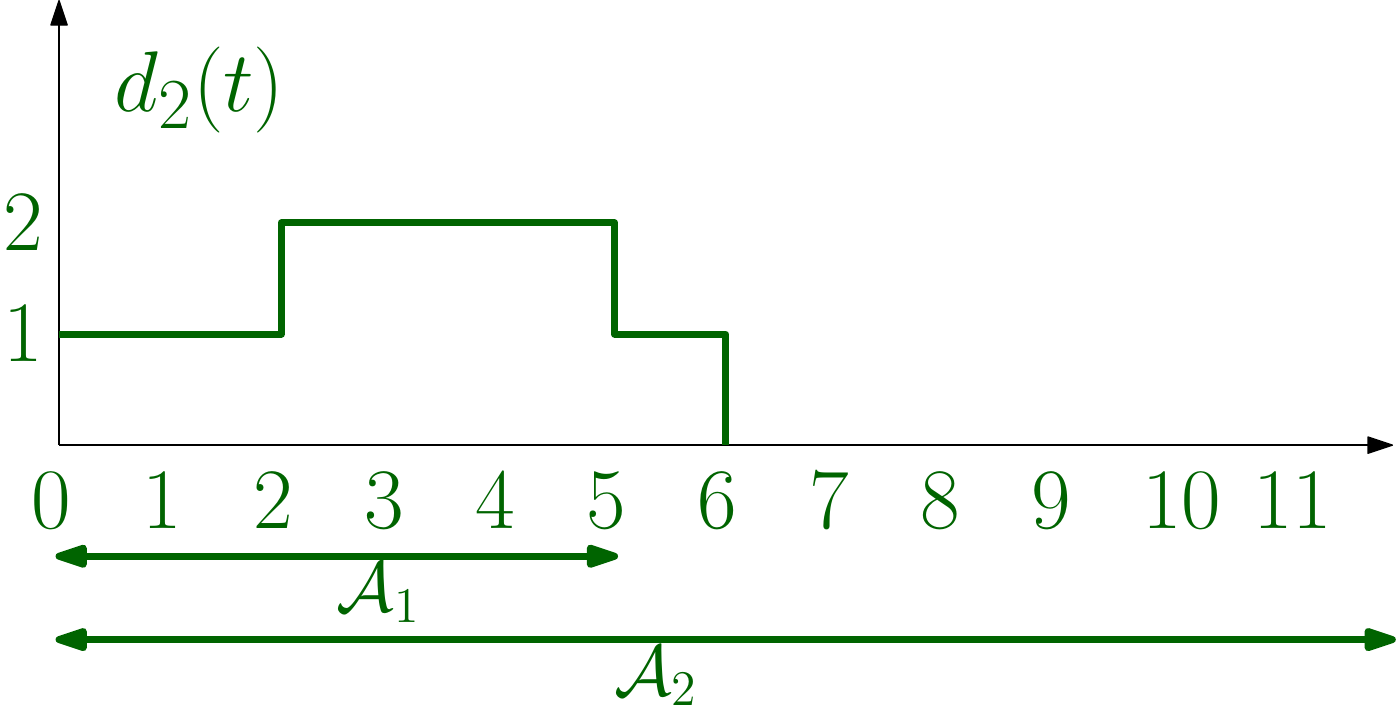}
\caption{Feasible power profiles $d_1$ and $d_2$}
\label{feaspow}
\end{figure} 
Notice that $d_1(t)$ and $d_2(t)$ coincide over the initial interval $[0,2]$.
However, it is their respective behaviours for $t \geq 2$ which determines who is supposed to deliver the first two power units. Specifically, $d_1(t)$ can only be fulfilled with device $1$ delivering power in the interval $[0,3]$ and device $2$ delivering on the interval $[5,11]$.
On the contrary, $d_2(t)$ can only be met if device $1$ delivers power over the interval $[2,5]$ and device $2$ in the interval $[0,6]$.
This proves that, even in elementary situations, causal dispatch policies do not exist.

In the light of our previous result, Theorem \ref{main}, it is worth noticing that the dispatch for $d_2(t)$ coincides with a fixed point $\bar{\lambda}=[0,0]'$. Indeed, the GGDDF policy (without availability windows) would initially allocate all the power to the second device (because of its higher discharge time), and additionally trigger battery $1$ in the interval $[2,5]$.
The profile $d_1(t)$ corresponds instead to $\bar{\lambda}=[6/7,0]'$.
Notice that, with such choice of $\bar{\lambda}$ we get
$\tilde{x}(0) = [3+ 7 \cdot 6/7, 6]' = [9,6]'$.
Hence, in $[0,3]$ only device $1$ is discharging, due to its highest discharge time, while at time $3$ the following equality $\tilde{x}_1 (3)=\tilde{x}_2(3)=6$ is achieved, so that
devices will then discharge together until time $11$, when they will both be empty.
The energy delivered outside the availability window by device $1$ is therefore $6$, which corresponds to $\lambda_1 = 6/\mu(\mathcal{T} \backslash \mathcal{A}_1) = 6/7$. Instead $\mathcal{A}_2 = \mathcal{T}$ and therefore $\lambda_2=0$.

\section{Time-domain characterization of feasible set}
In this Section, we present necessary and sufficient conditions on the aggregate demand signal $d(t)$ for feasibility with respect to a fleet $\mathcal{N}$ with given initial conditions, power ratings and availability sets.
Our main result states that a signal $d(t)$ is feasible if and only if it fulfills a set of linear constraints that can be explicitly computed as functions of initial conditions and availability windows.
\begin{theo}
	\label{necandsufficient}
	Consider a fleet $\mathcal{N}$ of batteries with maximum power ratings $\bar{P}_j$, availability windows $\mathcal{A}_j \subset \mathcal{T}$
	and initial time-to-discharge $x_j(0)$ respectively, for all $j \in \mathcal{N}$.
	A signal $d:\mathcal{T} \rightarrow [0,+\infty)$ is feasible for $\mathcal{N}$ if and only if, for all $\mathcal{W} \subset\mathcal{T}$ the following inequality holds:
	\begin{equation}
	\label{fundamentallimitation}
	\int_{\mathcal{W}} d(t) \, dt \leq \sum_{j \in \mathcal{N}} \min \{
	\mu( \mathcal{A}_j \cap \mathcal{W} ), x_j(0) \} \bar{P}_j.
	\end{equation}
\end{theo}
From a physical point of view, equation (\ref{fundamentallimitation}) requires that the
energy request over any time window $\mathcal{W}$ be less than what the fleet can deliver, over the same time window $\mathcal{W}$ by operating at rated power (for devices who have enough energy) or at any other rate that yields a depletion of the battery within the considered time window, for those who haven't. 
It is worth pointing out that the above conditions, in the case of $\mathcal{A}_j=
[0,+\infty)$ for all $j \in \mathcal{N}$ boil down to existing characterizations of feasibility. Indeed, equation (\ref{fundamentallimitation}) reads:
\begin{equation*}
\begin{aligned}
\int_{\mathcal{W}} d(t) \, dt \leq& \sum_{j \in \mathcal{N}} \min \{
\mu( \mathcal{W} ), x_j(0) \} \bar{P}_j \\
=&\sum_{j \in \mathcal{N}} \bar{P}_j \int_0^{\mu( \mathcal{W})}  
\textbf{1} (t)- \textbf{1} (t - x_j(0)) \, dt, 
\end{aligned}
\end{equation*}
with $\textbf{1}(\cdot)$ denoting the Heaviside function.
Specifically, for non-increasing demand signals $d(\cdot)$, the maximum integral value for a given set $\mathcal{W}$ of assigned measure $\mu(\mathcal{W})$, is achieved for
$\mathcal{W} = [0, \mu( \mathcal{W} )]$, viz. $\int_{\mathcal{W}} d(t) \, dt \leq \int_0^{\mu(W)} d(t) \, dt$.
Hence, conditions (\ref{fundamentallimitation}) can equivalently be stated (in the case of non-increasing signals), as:
\[ \int_{0}^T d(t) dt \leq  \int_0^T \sum_{j \in \mathcal{N}} \bar{P}_j   
[\textbf{1} (t)- \textbf{1} (t - x_j(0)) ] \, dt, \]
for all $T \in \mathcal{T}$ which is in agreement with the result in \cite{zachary}.

An important Corollary of Theorem \ref{necandsufficient} is for demand signals $d(t)$ which are piecewise constant on finitely many equally spaced time intervals, and for availability sets which are union of such equally spaced intervals.
This corresponds to $d(t)$ of the form $d(t) = \sum^{T}_{k=1} d_{k} \left[ \textbf{1} (t-k) - \textbf{1} (t - (k+1) ) \right]$
where $\textbf{1}(\cdot)$ is the Heaviside function. From a practical point of view this correspond to a \emph{discrete} time formulation of the dispatch problem.
Let us denote the time-horizon as $\mathcal{T} = \{1,2,\ldots,T\}$ and the availability windows $\mathcal{A}_j\subseteq \mathcal{T}$.
\begin{corollary}
	Consider a fleet $\mathcal{N}$ of batteries with maximum power ratings $\bar{P}_j$, availability windows $\mathcal{A}_j \subset \mathcal{T}$
	and initial time-to-discharge $x_j(0)$ respectively, for all $j \in \mathcal{N}$.
	A signal $d:\mathcal{T} \rightarrow [0,+\infty)$ is feasible for $\mathcal{N}$ if and only if, for all $\mathcal{W} \subset\mathcal{T}$ the following inequality holds:
	\begin{equation}
	\label{fundamentallimitationdt}
	\sum_{k \in \mathcal{W}} d(k)  \leq \sum_{j \in \mathcal{N}} \min \{
	\emph{\textrm{card}}( \mathcal{A}_j \cap \mathcal{W} ), x_j(0) \} \bar{P}_j.
	\end{equation}
\end{corollary}
It is worth pointing out that inequalities (\ref{fundamentallimitationdt}) completely characterize the set of feasible demand profiles for any fleet $\mathcal{N}$ with arbitrary availability sets. For a time-horizon of $T$ sampling intervals, $2^T$ constraints are enough to characterize the polytope of feasible aggregated demand profiles, regardless of the size of the fleet.

\section{On optimality of dispatch policies}
Throughout this Section, we highlight important optimality properties of the dispatch policies previously introduced.
We defined already the set of feasible demand profiles for a fleet subject to availability constraints, according to equation (\ref{feset}).
We define a similar notion, for fleets without full availability, according to:
\begin{equation}
\label{feset2}
\begin{array}{rcl}
\tilde{\mathcal{F}} ( E(0) ) &:=& \Big \{ d(\cdot): [0,+\infty)  \rightarrow [0,+ \infty): \\
& & \exists \{ u_i : [0,+\infty) \rightarrow [0,\bar{P}_i] \}_{i \in \mathcal{N}}: \\
& &E_i(0)- \int_0^t u_i(\tau) d \tau \geq 0, \,\, \forall \,\, t \geq 0  \\
& & d(t) = \sum_{i \in \mathcal{N}} u_i(t) \Big \}.
\end{array}
\end{equation}
One important feature of the GGDDF policy in the case of \emph{full} device availability, is that it maximizes future flexibility, viz. the set of feasible demand profiles is as large as possible, when compared with respect to set-inclusion.  
A similar result also holds in this case, provided flexibility is measured on a fleet without availability constraints, and only at the end of the considered time horizon.
\begin{theo}
	\label{optimalfeas}
		Consider a fleet of devices $\mathcal{N}$ with initial time-to-discharge, power ratings and availability windows $x_j(0)$, $\bar{P}_j$, and $\mathcal{A}_j$ respectively, for $j \in \mathcal{N}$.  
	Define, for convenience, the corresponding energy vector $E_j (0) = x_j(0) \bar{P}_j$ and let $d(\cdot):
	\mathcal{T} \rightarrow [0, + \infty)$ be a feasible aggregated demand signal, viz. $d \in \mathcal{F} (E(0))$. Let $\bar{\lambda}$ be a fixed point of the associated map $\Lambda$
	and $\tilde{x} (0)$ be the corresponding vector of time-to-discharge, viz.
	$\tilde{x}_j (0) = x_j(0) + \bar{ \lambda}_j \mu( \mathcal{T} \backslash \mathcal{A}_j)$.
	Denote by $\bar{\tau} = \max_{t \in \mathcal{T}} t$. Then, for any choice of $u(t)$ such that $d(t) = \sum_{j: t \in \mathcal{A}_j} u_j(t) = d(t)$ and $d(t) \in [0, \bar{P}_j)$, the solution $E(t)$ of (\ref{popequation}) fulfils:
\[ \tilde{\mathcal{F}} ( E(\bar{\tau}) ) \subseteq \tilde{\mathcal{F}} ( \Omega^{-1} \, \varphi ( \bar{\tau},\tilde{x}(0),d   ) ). \]
\end{theo}
\begin{proof}
Consider any feasible policy $u(t)$, which fulfils demand $d(t)$ over 
$\mathcal{T}$. This can be made into a feasible policy $\tilde{u}$ for aggregated demand $\tilde{d}$ and on the same fleet with complemented energy and disregarding availability constraints, according to:
\[  \tilde{u}_j(t) = \left \{ \begin{array}{rl} u_j(t) & t \in \mathcal{A}_j \\ K (t, \varphi(t,\tilde{x}(0), d(\cdot)),d(t)) & t \notin \mathcal{A}_j \end{array}    \right . \]
Let $\tilde{E} (t)$ be the solution corresponding to $\tilde{u}(t)$ and 
for initial condition $\tilde{E}(0) = \Omega^{-1} \tilde{x}(0)$. 
By the optimality of GGDDF policies with respect to feasibility of future demand signals, we have
$ \tilde{\mathcal{F}} ( \tilde{E} (t) ) \subseteq 
\tilde{\mathcal{F}} (  \Omega^{-1}  \tilde{\varphi}(t, \tilde{x}(0), \tilde{d}(\cdot)) ), \, \forall \, t \in \mathcal{T}$.
In particular, for $t = \bar{\tau}$ we have,
$\tilde{E} (\bar{\tau}) = E( \bar{\tau})$ (thanks to the fixed point condition and definition of $\tilde{E}$) and therefore, by exploiting
$\tilde{\varphi}(t, \tilde{x}(0), \tilde{d}(\cdot)) ) = \varphi(t, \tilde{x}(0),d(\cdot) )$ we see that
$\tilde{\mathcal{F}} ( E (\bar{\tau}) ) \subseteq 
\tilde{\mathcal{F}} (  \Omega^{-1}  \varphi(\bar{\tau}, \tilde{x}(0), d(\cdot)) )$,
which proves the claim.
\end{proof}
Theorem \ref{optimalfeas} shows that the set of feasible power demand signals achieved by the policy (\ref{policy}) at the end of the prediction horizon $\mathcal{T}$ is maximal, with respect to set-inclusion, and in regard of any other feasible power schedule, when formulated for an idealized fleet without availability constraints. 

The notion of unserved energy is an important measure of reliability of supply. From the mathematical point of view, in the context of this paper, we may define, for each policy $u(t)$ defined over $\mathcal{T}$ the following functional:
\begin{equation}
    \label{unservedfunctional}
\mathcal{U} (u(\cdot) ) := \int_{\mathcal{T}}  \max  \Big \{ d(t) - \sum_{j \in \mathcal{N}} u_j(t), 0 \Big \} \, dt. 
\end{equation}
When a demand signal is feasible, one might find $u(\cdot)$ such that the corresponding unserved energy is zero (and this is what the discussed policies allow to do), however, for unfeasible demand signals, it is still desirable to minimize $\mathcal{U}$.
The problem can be formulated as follows:
\begin{equation}
\label{minuns}
\begin{array}{rl}
\min_{u(\cdot), E(\cdot)}& \mathcal{U} ( u(\cdot)) \\
\textrm{s.t.}\quad \dot{E}(t) &= - u(t) \quad \forall \, t \in \mathcal{T} \\
E(0) &= \Omega^{-1} x(0)  \\
0 \leq u_j(t) & \leq \bar{P}_j \quad \forall \, t \in \mathcal{T}, \; \forall \, j \in \mathcal{N} \\
E(t) & \geq 0 \quad \forall \, t \in \mathcal{T} \\ 
u_j(t) &= 0 \quad  \forall \, t \notin \mathcal{A}_j, \; \forall j  \in \mathcal{N}.
\end{array}
\end{equation} 
The solution to problem (\ref{minuns}) can be found by considering an augmented demand signal $\tilde{d}$ for a fleet with full availability and suitably augmented initial time-to-discharge $\tilde{x} (0)$.
Consider the modified cost functional:
\begin{equation}
    \label{unservedfunctional2}
\tilde{\mathcal{U} } (\tilde{u}(\cdot) ) := \int_{\mathcal{T}}  \max  \Big \{ \tilde{d}(t) - \sum_{j \in \mathcal{N}} \tilde{u}_j(t), 0 \Big \} \, dt,
\end{equation}
where $\tilde{d}$ is defined according to (\ref{auxdemand}). Its minimization disregarding partial availability constraints can be formulated as: 
\begin{equation}
\label{minuns2}
\begin{array}{rl}
\min_{\tilde{u}(\cdot), \tilde{E}(\cdot)}& \mathcal{U} ( \tilde{u}(\cdot)) \\
\textrm{s.t.}\quad \dot{\tilde{E}}(t) &= - \tilde{u}(t) \quad \forall \, t \in \mathcal{T} \\
\tilde{E}(0) &= \Omega^{-1} \tilde{x}(0)  \\
0 \leq \tilde{u}_j(t) & \leq \bar{P}_j \quad \forall \, t \in \mathcal{T}, \; \forall \, j \in \mathcal{N} \\
\tilde{E}(t) & \geq 0 \quad \forall \, t \in \mathcal{T}.
\end{array}
\end{equation} 
Previous results (see \cite{ETA-tops}) show that Problem (\ref{minuns2}) can be solved through a GGDDF policy. Our main result in this respect is to connect the optimal solution of (\ref{minuns}) to that of (\ref{minuns2}).
\begin{theo}
\label{minimumenergyth}
Consider the optimisation problems in (\ref{minuns}) and (\ref{minuns2}) where
the augmented demand $\tilde{d}$ is defined according to (\ref{auxdemand}) and 
the initial condition $\tilde{x}(0)$ fulfils  $\tilde{x}_j (0) = x_j(0) + \bar{ \lambda}_j \mu( \mathcal{T} \backslash \mathcal{A}_j)$ for some fixed point $\bar{\lambda}$ of $\Lambda(\cdot)$.
Then, the minimum value of unserved energy for (\ref{minuns}) equals the minimum of
(\ref{minuns2}).
\end{theo}

\begin{proof}
As a first step it is useful to remark that minimising unserved energy can be equivalently formulated as:
\begin{equation}
\label{minuns3}
\begin{array}{rl}
\min_{u(\cdot), E(\cdot)} & - \sum_{j \in \mathcal{N}}  \min \{ E(\bar{\tau}),0 \} \\
\textrm{s.t.}\quad \dot{E}(t) &= - u(t) \quad \forall \, t \in \mathcal{T} \\
E(0) &= \Omega^{-1} x(0)  \\
0 \leq u_j(t) & \leq \bar{P}_j \quad \forall \, t \in \mathcal{T}, \; \forall \, j \in \mathcal{N} \\
u_j(t) &= 0  \quad\forall \, t \notin \mathcal{A}_j,  j \in \mathcal{N}, \\
d(t)& =\sum_{j \in \mathcal{N}} u_j(t)
\end{array}
\end{equation} 
where $[0,\bar{\tau}]=\mathcal{T}$.
Similarly, for fleets without partial availability constraints, problem (\ref{minuns2}) is equivalently written as:
\begin{equation}
\label{minuns4}
\begin{array}{rl}
\min_{\tilde{u}(\cdot), \tilde{E}(\cdot)}& - \sum_{j \in \mathcal{N}} \min  \{  \tilde{E}_j ( \bar{\tau}) ,0 \}  \\
\textrm{s.t.} \quad \dot{\tilde{E}}(t) &= - \tilde{u}(t) \quad \forall \, t \in \mathcal{T} \\
\tilde{E}(0) &= \Omega^{-1} \tilde{x}(0)  \\
0 \leq \tilde{u}_j(t) & \leq \bar{P}_j \quad \forall \, t \in \mathcal{T}, \; \forall \, j \in \mathcal{N}. \\
\tilde{d}(t)& =\sum_{j \in \mathcal{N}} \tilde{u}_j(t)
\end{array}
\end{equation} 
Notice that (\ref{minuns3}) and (\ref{minuns4}) relax the positivity constraint on $E$ and $\tilde{E}$ respectively. In its place, an equality constraint is added, forcing the total power delivered to meet aggregate demand. In this way, input signals that push charge levels to negative values are still regarded as feasible, but their impact is accounted for in the cost functional. 
The result is proved through a series of inequality.
Consider the subset of signals  $\tilde{u}$ achieved through the following construction:
\begin{equation}
    \label{tildeupara}
    \tilde{u}_j (t) = \left \{ \begin{array}{rl} u_j(t) & \textrm{if } t \in \mathcal{A}_j \\ K_j ( t, \varphi(t,\tilde{x}(0),d(\cdot)), d(t) ) & \textrm{if } t \notin \mathcal{A}_j. \end{array} \right. 
    \end{equation}
where $u$ is the new decision variable.
The minimum value of (\ref{minuns3}) (with respect to $u$) over this restricted class of signals is greater or equal to the minimum value of unserved energy in (\ref{minuns2}). Furthermore, since the GGDDF policy is optimal for (\ref{minuns2})
and results in an optimal input policy $\tilde{u}_j^*$ which fulfils
$\tilde{u}_j^*(t)=  K_j ( t, \varphi(t,\tilde{x}(0),d(\cdot)), d(t) )$ for $t \notin \mathcal{A}_j$, the minimum of (\ref{minuns3}) with respect to the restricted class of signals is equal to the minimum value of unserved energy in (\ref{minuns2}).
On the other hand, because of the fixed point condition, the amount of energy delivered by any $\tilde{u}$ parameterized according to (\ref{tildeupara}) outside the availability windows exactly matches the extra energy allowed to devices at time $0$. Hence,
$E ( \bar{\tau} ) = \tilde{E} (\bar{\tau})$, when $E(t)$ is discharged according to equation $\dot{E}=-u$, and $u_j(t)=0$ for $t \notin \mathcal{A}_j$.
Hence the minimum value of (\ref{minuns3}) over the restricted class of $\tilde{u}$ signals equals the minimum value of (\ref{minuns4}).
\end{proof}

When avoiding energy curtailment is to be prioritized, another functional of interest is the so called \emph{time-to-failure}. 
For each dispatch policy $u$ of aggregate demand $d(\cdot)$ this is defined as the first time some battery state-of-charge becomes negative.
Formally,
$\mathcal{T}_f ( x(\cdot) ) := \inf \{ t \in \mathcal{T}: \min_{j \in \mathcal{N}} x_j(t) < 0 \}$.

Our aim is to solve the following optimisation problem:
\begin{equation}
\label{timefailmax}
\begin{array}{rl}
\tau^{*}:= \max_{u(\cdot),x(\cdot)} & \mathcal{T}_f (x(\cdot)) \\
\textrm{s.t.} \quad x(0) &= x_0 \\
\dot{x}(t) &= - \Omega u(t), \quad \forall \, t \in \mathcal{T} \\
0 \leq u_j(t) & \leq \bar{P}_j, \quad \forall \, t \in \mathcal{T}, \forall \, j \in \mathcal{N} \\
u_j(t) & = 0 \quad \forall \, t \notin \mathcal{A}_j, \, \forall \, j \in \mathcal{N} \\
d(t)  & = \sum_{j \in \mathcal{N}} u_j(t). 
\end{array}
\end{equation}

For fleets without availability constraints this can be maximized through the use of a GGDDF policy. However, it is shown in the following Section that the naive application of this same policy does not achieve maximisation of  time-to-failure. 
We propose an iterative procedure to compute the optimal time-to-failure and the associated dispatch.
\begin{enumerate}
\item Let $\mathcal{T}=[0,\tau_0]$; Let $k=0$;
\item Repeat:
\begin{itemize}
    \item Compute $\bar{\lambda}_k$, fixed point of $\Lambda$, over window
    $[0,\tau_k]$;
    \item Apply GGDDF policy from $\tilde{x}(0)$. Let $\tilde{x}_k$ be corresponding state evolution
    \item $\tau_{k+1} := \mathcal{T}_f ( \tilde{x}_k )$; increase $k$;
\end{itemize} 
\end{enumerate}
Our main result is the following:
\begin{theo}
\label{timefailuremax}
The iteration defined above converges to the optimal time-to-failure, i.e., $\lim_{k \rightarrow + \infty} \tau_k = \tau^*$, as in (\ref{timefailmax}).
\end{theo}

\section{Numerical Examples}

In this section, the proposed dispatch algorithm for heterogeneous devices is applied to some numerical case studies. The computational tasks were implemented using MATLAB R2019a and solved by the routine \textit{fsolve} or \textit{fmincon}, on a computer with 2-core 3.50GHz Intel(R) Xeon(R) E5-1650 processor and 32GB RAM.

\subsection{Feedback policy}\label{example:policy}

Next, the proposed feedback policy is tested on a fleet of $500$ devices. A time interval of $24$ hours, from $12:00$ h to $12:00$ h of the next day, is considered. For every device $j \in \mathcal{N} = \{1, 2, \cdots, 500\}$, the rated power $\bar{P}_{j}$ is set to the same value $\bar{P}_{j} = 1$ KW. The initial time-to-discharge $x_{j}(0)$ (initial energy $E_{j}(0)$) follows a normal distribution with mean $\mu_{E} = 8$ kWh and standard deviation $\sigma_{E} = 1.5$ kWh. It is assumed that each device can discharge only within a continuous time interval $[t_{j}, t_{j} + d_{j}]$ h, where $t_{j}$ and $d_{j}$ follow normal distributions, with the following mean and standard deviation:
$\mu_{t} = 18:00 \,\,\text{h}, \,\, \sigma_{t} = 1 \,\,\text{h}, \,\, \mu_{d} = 10 \,\,\text{h}, \,\, \sigma_{d} = 2 \,\,\text{h}$.

To provide a clearer demonstration, the boundaries of availability windows are chosen as integers and we consider piecewise constant aggregate demand signals with integer switching time instants. The calculation of fixed $\bar{\lambda}$ has been completed in about $30$ minutes. Table \ref{table:computation} presents how the computational time for $\bar{\lambda}$ changes with respect to the number of devices in the fleet. It is observed that the time to obtained the fixed $\bar{\lambda}$ is growing approximately as $N^4$, thus exhibiting good scalability properties. When the proposed approach is adopted into the receding horizon framework, the convergence of $\bar{\lambda}$ will be faster because a warm start is available if there is no significant mismatch between the realisation and prediction. Moreover, in real-world applications, more powerful machines can further reduce the computational time and easily perform the algorithm on larger number of devices.

\begin{table}[H]
\vspace{-3mm}
\centering
\caption{Computational complexity of solving the fixed $\bar{\lambda}$}
\label{table:computation}
\vspace{-2mm}
\begin{tabular}{@{}cccccc@{}}
\toprule
\multicolumn{1}{c}{Device number} & 10  & 50 & 100 & 250 & 500\\ 
\midrule
\multicolumn{1}{c|}{Computation time (min)}  & $1\times10^{-4}$  & 0.002 & 0.05 & 1.51 & 30.13 \\ 
\bottomrule
\end{tabular}
\end{table}

Fig. \ref{Availability_Demand_Feasible} shows the aggregate availability, viz. $\sum_{j:t \in \mathcal{A_j}} \bar{P}_j$ and required demand profile. The demand and availability profiles satisfy all the necessary conditions discussed in the conference paper \cite{cdcbatteries}, and this demand is also feasible since the auxiliary time-to-discharge of all devices, in Fig. \ref{Time_to_Discharge_Feasible}, are non-negative.

\begin{figure}[!ht]
\centerline{\includegraphics[width=0.4\textwidth]{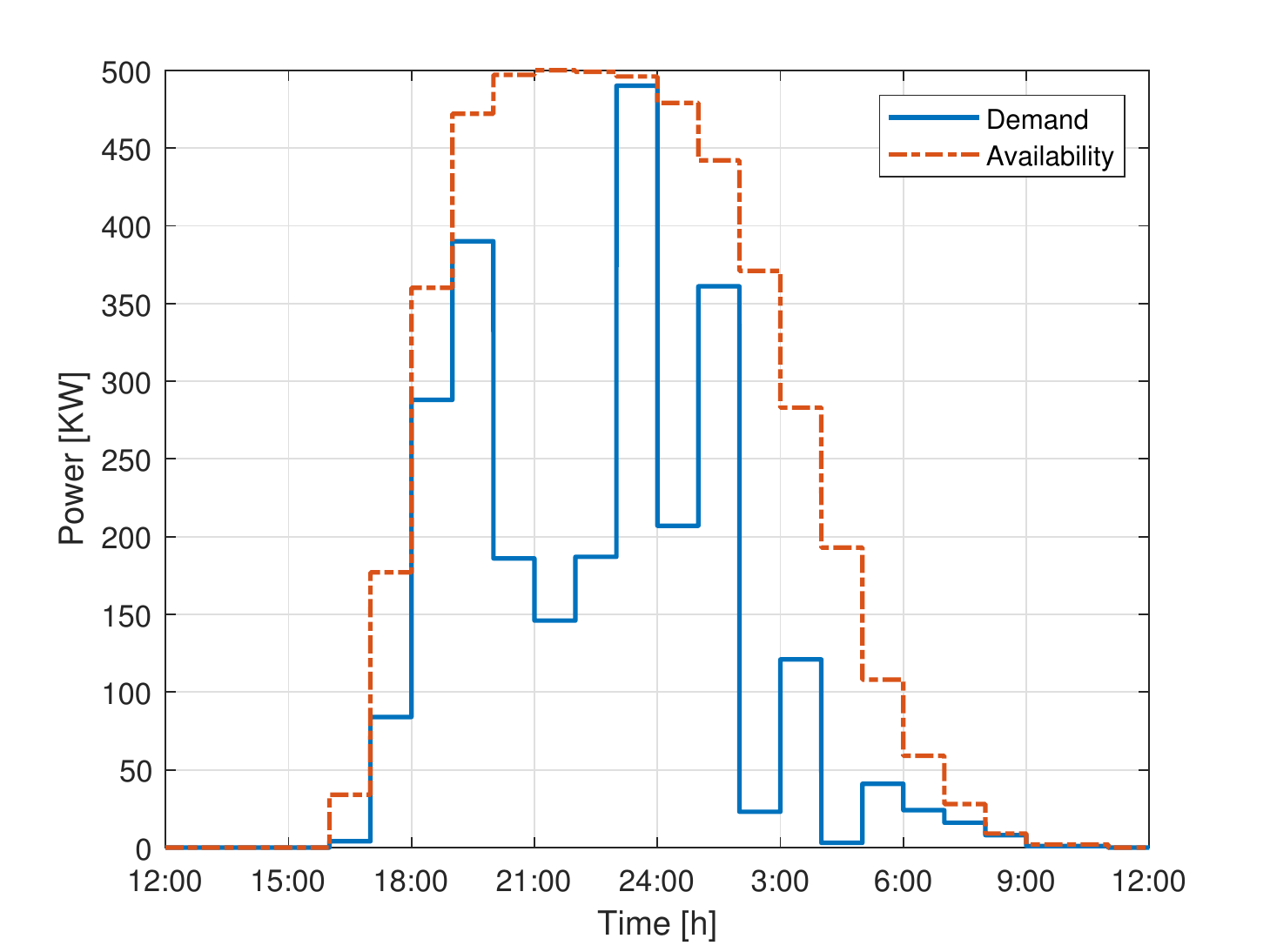}}
\vspace{-3mm}
\caption{Available aggregate power and a feasible demand profile.}
\label{Availability_Demand_Feasible}
\end{figure}

\begin{figure}[!ht]
\vspace{-5mm}
\centerline{\includegraphics[width=0.4\textwidth]{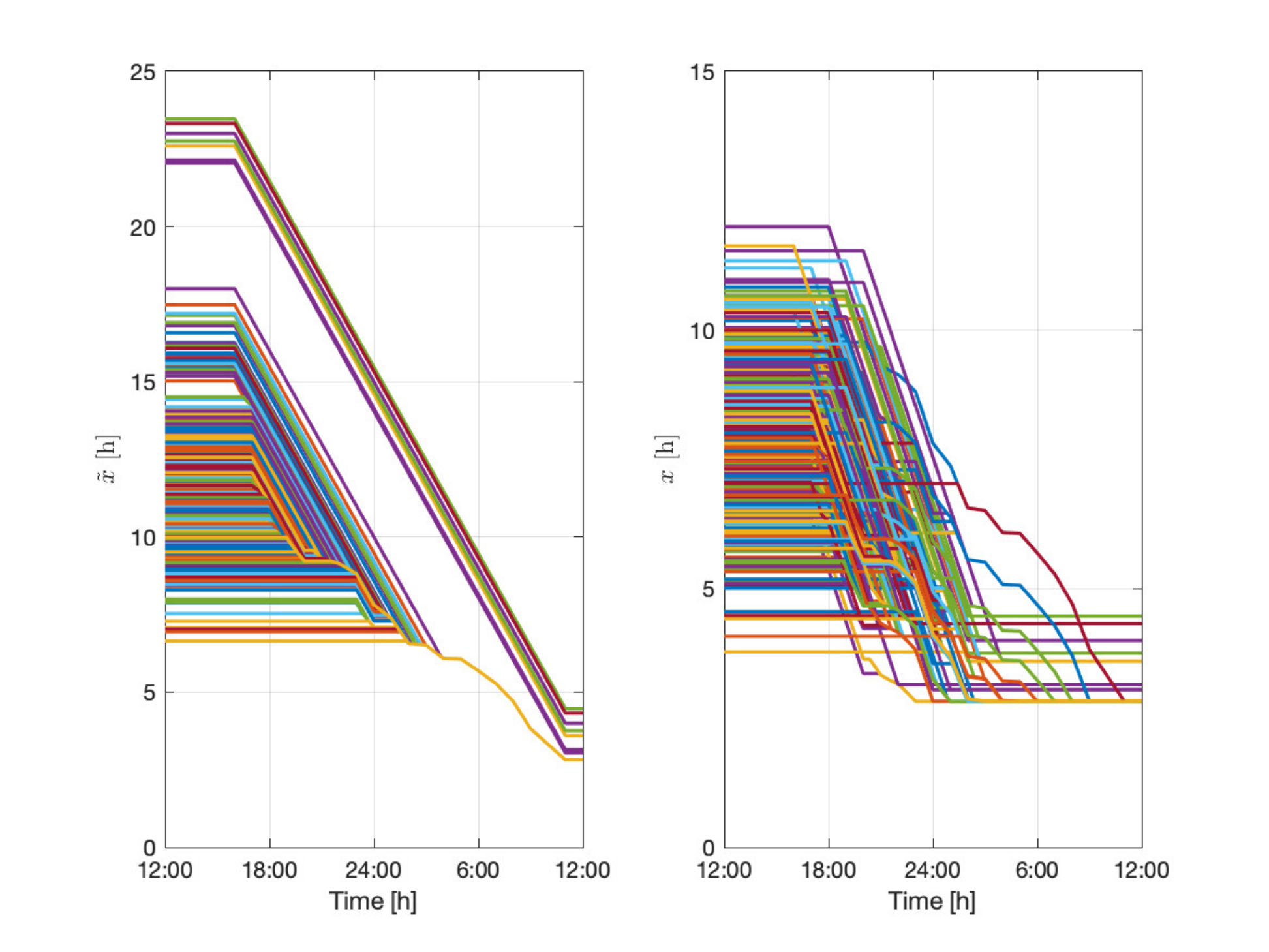}}
\vspace{-3mm}
\caption{Time-to-discharge for feasible demand profile with/without augmented energy.}
\label{Time_to_Discharge_Feasible}
\end{figure} 


Some representative examples of individual power dispatch signals are shown in Fig. \ref{Storage_Feasible}. Their availability is displayed with a green area and individual discharging power profiles are decided according to the ranking of auxiliary time-to-discharge which is presented in Fig. \ref{Time_to_Discharge_Feasible}. When the demand $d(t)=0$ before $16:00$ h and after $11:00$ h, all devices have a constant time-to-discharge. If the demand is positive, higher time-to-discharge devices are prioritized for discharging with $\bar{P}$, medium time-to-discharge devices are discharging at a fraction of their rated power, and lower time-to-discharge devices are controlled to preserve energy for later use.

\begin{figure}[!ht]
\centerline{\includegraphics[width=0.4\textwidth]{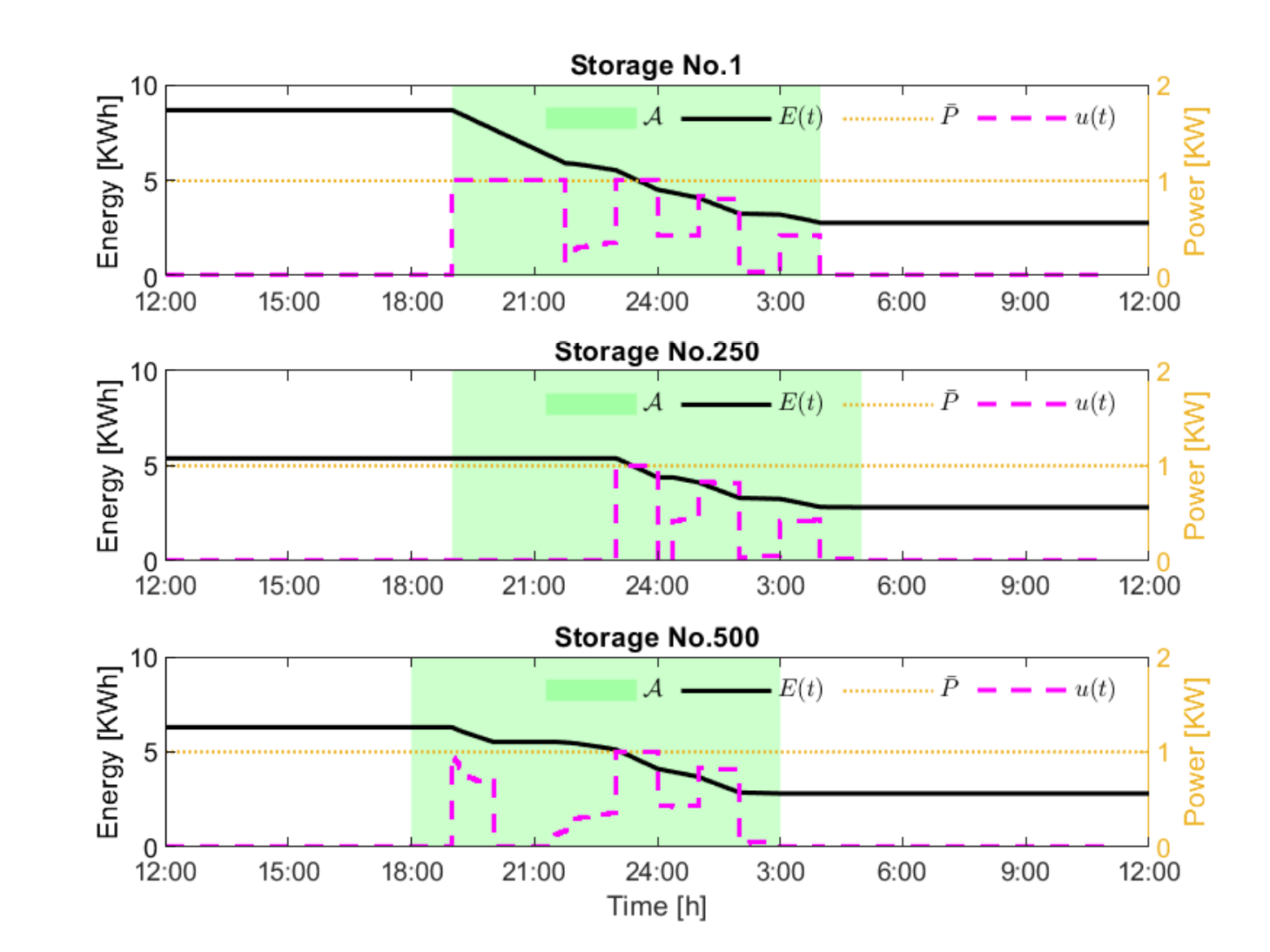}}
\vspace{-3mm}
\caption{Energy profiles $E(t)$ for different storage devices (subscript $i$ is neglected), together with availability interval $\mathcal{A}$, discharging profiles $u(t)$.}
\label{Storage_Feasible}
\end{figure}


\subsection{Minimum unserved energy vs. maximum time-to-failure}

We consider a smaller number of devices $N = 20$ with parameters following the same probability density functions as in Section \ref{example:policy}. The maximum aggregate availability of these devices and the required demand profile are shown in Fig. \ref{Availability_Demand_Infeasible}. Numerically, the aggregate initial energy is $\sum_{j \in \mathcal{N}} E_{j}(0) = 163.25$ kWh and the total energy by the demand profile is $\int_{\mathcal{T}} d(t) dt = 168.05$ kWh. As a result, this demand profile is unfeasible due to insufficient total energy, which leads to the consequence that some devices finish the discharging task with negative time-to-discharge according to the policy \eqref{policy}.

Fig. \ref{State_MinUE_vs_MaxT2F} demonstrates the auxiliary and actual time-to-discharge using two dispatch policies. Over the $24$ hours full time window, the GGDDF policy serves total energy $151.40$ kWh, hence, the minimum unserved energy is $\mathcal{U}(\cdot) = 16.65$ kWh. Nevertheless, this policy has a relatively short time-to-failure, since one device crosses into negative time-to-discharge at $24:30$ h on the left-bottom subplot of Fig. \ref{State_MinUE_vs_MaxT2F}. 

We compare the above scheduling with the policy resulting from the iteration algorithm corresponding to Theorem \ref{timefailuremax}.
This identifies the discharging schedules for all devices up to around $4:15$ h, which is the maximum time-to-failure. This is significantly larger (approx. $3:45$ h) than what is achieved by the previous dispatch. Notice that the total energy served before this time instant is $149.69$ kWh. 
\comment{
Furthermore, as shown in the bottom of Fig. \ref{RealState_MinUE_vs_MaxT2F}, there are still five devices with positive energy at $4:15$ h. To analyze the amount of energy these non-empty devices can contribute to the demand requirement after $4:15$ h, one can see that only one device with instantaneous energy $1.71$ kWh is available for discharging afterwards, as depicted in Fig. \ref{DischargingAvailableDevice}. The time duration until its end of availability window is $1.75$ hours which is sufficient for this device to release all its energy. Accordingly, the total energy all devices can provide after $4:15$ h is $1.71$ kWh, which implies that maximum energy following the second dispatch policy is $151.40$ kWh over the full time window. Therefore, the schedule provided by the maximum time-to-failure policy can be completed to a schedule achieving minimum unserved energy over the full window $\mathcal{T}$. 
We conjecture this might be true in general, but we leave it as an open question for future research. The converse is instead not true, as clearly shown by}
The considered example exhibits a significant gap between the time-to-failure of the original minimum unserved energy schedule and the maximum time-to-failure achievable. 
This is in contrast to the case of full availability where the GGDDF policy achieve both maximum time-to-failure and minimum unserved energy.

\begin{figure}[!ht]
\centerline{\includegraphics[width=0.4\textwidth]{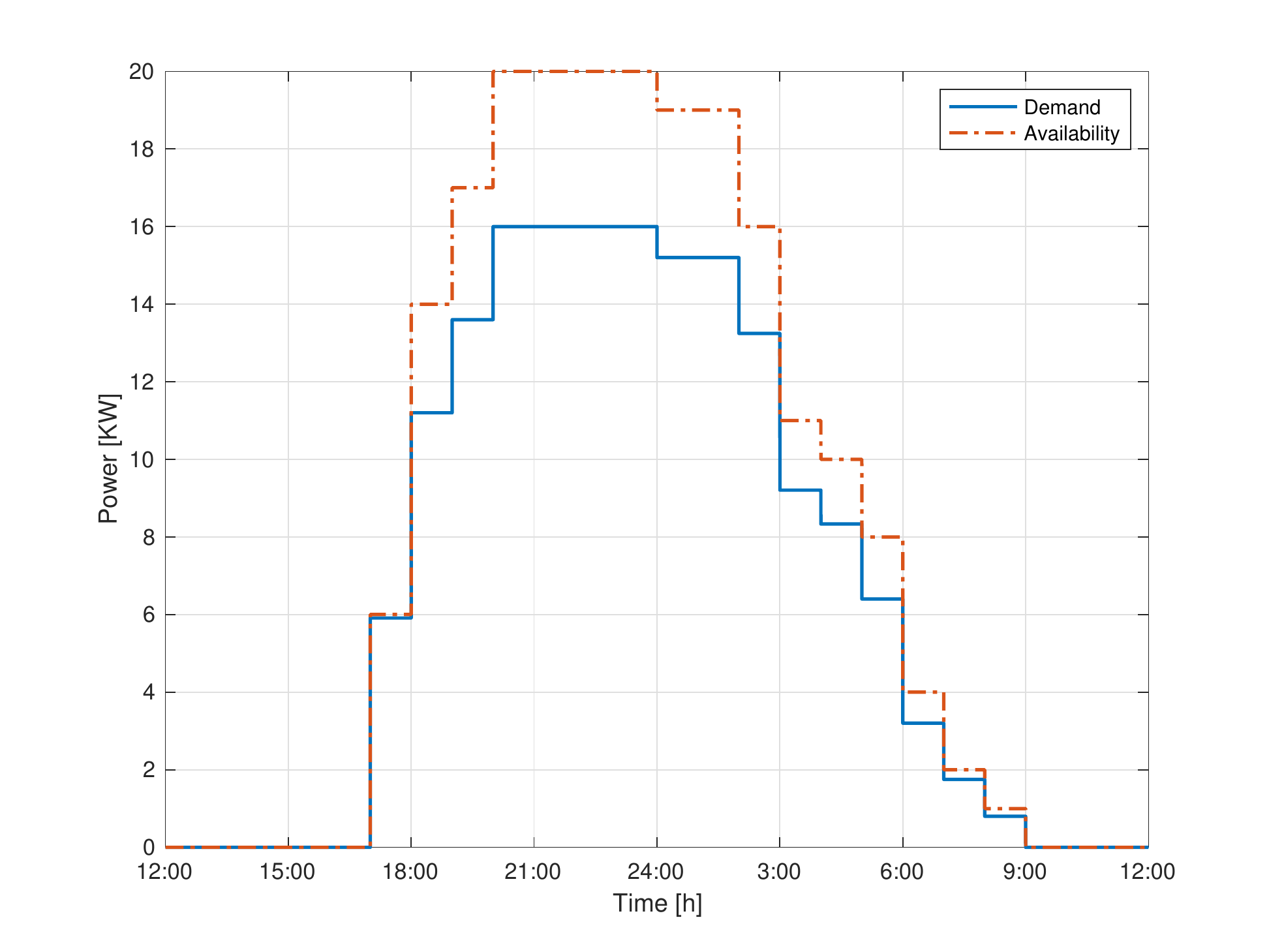}}
\vspace{-3mm}
\caption{Available aggregate power and an unfeasible demand profile.}
\label{Availability_Demand_Infeasible}
\end{figure}

\begin{figure}[!ht]
\centerline{\includegraphics[width=0.4\textwidth]{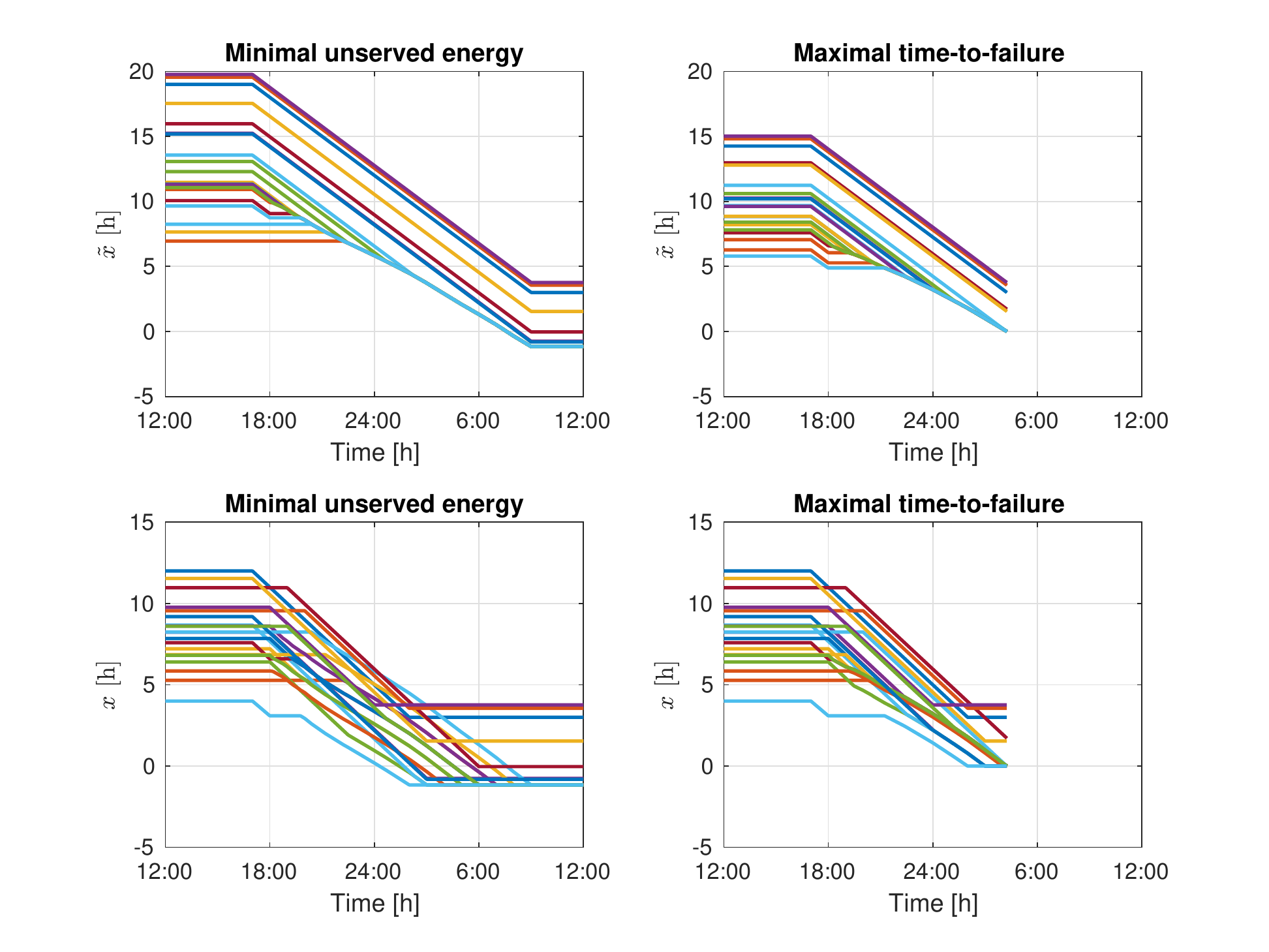}}
\vspace{-3mm}
\caption{Time-to-discharge for the unfeasible demand profile.}
\label{State_MinUE_vs_MaxT2F}
\end{figure}






\section{Conclusions}
This paper solves the optimal dispatch problem for heterogeneous fleets of storage devices, subject to partial availability constraints
and without cross-charging. This significantly extends previous known results which were limited to fleets with full availability, \cite{ETA-pscc,ETA-tosg,ETA-tops,zachary}.
The approach transforms the problem to one of dispatch of a fleet over the same time-horizon but in the absence of availability constraints, for some auxiliary and increased demand signal and correspondingly increased initial energies (or discharge time).
Admissible dispatch policies are provided whenever the aggregate demand signal is feasible and policies maximizing time-to-failure or minimizing unserved energy are formulated and discussed, for the case of unfeasible demand.

A characterization of the feasible set of aggregate demand signals is presented, which may serve as an effective computational approach to embed flexible demand or coordinated fleet operation during outages in large scale optimisation problems, by exactly capturing the degrees of freedom afforded by the fleet without explicit mention of individual power schedules. 
Finally, examples of application of the techniques are provided for fleets of medium and large size, to demonstrate the effectiveness and scalability of the approach. 
Several open questions remain in this area, particularly related to more realistic battery models, multi-area dispatch problems or bidirectional power transfers.

\appendices

\section{Proof of Lemma \ref{keylemma} }
\label{bappendix}
We prove the lemma by separately considering all the cases involved in the policy definition.
\begin{enumerate}
	\item Let $i \in \mathcal{N}_{\tau_{k}}$, with  
	$\sum_{h \leq k} \sum_{j \in \mathcal{N}_{\tau_{h}}: t \in \mathcal{A}_j} \bar{P}_j
	\leq d$.
	Then $K_i (t,x,d) = \bar{P}_i$, moreover, $K_j(t,x,d)= \bar{P}_j$ for all $j \in \mathcal{N}_{\tau_{h}}$ with $h \leq k$. Hence:
	\begin{equation*}
	    \begin{aligned}
    \sum_{h \leq k}& \sum_{j \in \mathcal{N}_{\tau_{h}}} \bar{P}_j \\
	& = \sum_{h \leq k} \left ( \sum_{j \in \mathcal{N}_{\tau_{h}}: t \in \mathcal{A}_j} \hspace{-2mm} \bar{P}_j + \sum_{j \in \mathcal{N}_{\tau_{h}}
		: t \notin \mathcal{A}_j} \hspace{-2mm} K_j (t,x,d)  \right ) \\&\leq d +  \sum_{j 
		: t \notin \mathcal{A}_j} \hspace{-2mm} K_j (t,x,d).
	\end{aligned}
	\end{equation*}
	By definition of $\tilde{K}$ this implies 
	$\tilde{K}_i \left ( x , d +  \sum_{j 
		: t \notin \mathcal{A}_j} K_j (t,x,d) \right ) = \bar{P}_i = K_i(t,x,d)$.
	\item 	Let $ i \in \mathcal{N}_{\tau_{k}}$ with
	$ \sum_{h < k} \sum_{j \in \mathcal{N}_{\tau_{h}}: t \in \mathcal{A}_j} \bar{P}_j
	\leq d <  \sum_{h \leq k} \sum_{j \in \mathcal{N}_{\tau_{h}}: t \in \mathcal{A}_j} \bar{P}_j$.
	Then $K_j (t,x,d) = r(t) \bar{P}_j$ for all $j \in \mathcal{N}_{\tau_{k}}$, and
	$K_j (t,x,d) = \bar{P}_j$ for all $j \in \mathcal{N}_{\tau_{h}}$ with $h<k$. In addition $K_j (t,x,d) = 0$ for all $j \in \mathcal{N}_{\tau_{h}}$ with $h>k$.
	As a consequence:
	\begin{equation*}
    \begin{aligned}
	\sum_{h < k} \sum_{j \in \mathcal{N}_{\tau_{h}}} \bar{P}_j =& \sum_{h < k} \left ( \sum_{j \in \mathcal{N}_{\tau_{h}}: t \in \mathcal{A}_j} \bar{P}_j +  \sum_{j \in \mathcal{N}_{\tau_{h}}: t \notin \mathcal{A}_j} \bar{P}_j \right ) \\
	\leq& d +  \sum_{h < k} \sum_{j \in \mathcal{N}_{\tau_{h}}: t \notin \mathcal{A}_j} \bar{P}_j  \\ =&
	d +  \sum_{h < k} \sum_{j \in \mathcal{N}_{\tau_{h}}: t \notin \mathcal{A}_j} K_j (t,x,d) \\
	\leq& d + \sum_{j : t \notin \mathcal{A}_j}  K_j (t,x,d).
	\end{aligned}
	\end{equation*}
	Moreover:
	\begin{equation*}
    \begin{aligned}
	d +& \sum_{j : t \notin \mathcal{A}_j} K_j (t,x,d) \\ <&  \left ( \sum_{h \leq k} \sum_{j \in \mathcal{N}_{\tau_{h}}: t \in  \mathcal{A}_j} \hspace{-2mm}\bar{P}_j \right ) + \sum_{j : t \notin \mathcal{A}_j} K_j (t,x,d)  \\
	= & \left ( \sum_{h \leq k} \sum_{j \in \mathcal{N}_{\tau_{h}}: t \in  \mathcal{A}_j} \hspace{-2mm}\bar{P}_j \right ) + \left ( \sum_{h < k} \sum_{j \in \mathcal{N}_{\tau_{h}}: t \notin \mathcal{A}_j} \hspace{-2mm} \bar{P}_j \right )\\ &\,\, +\hspace{-2mm} \sum_{j \in \mathcal{N}_{\tau_{k}}: t \notin \mathcal{A}_j} \hspace{-2mm} r(t) \bar{P}_j   \\
	= &  \left ( \sum_{h < k} \sum_{j \in \mathcal{N}_{\tau_{h}}} \bar{P}_j \right )
	+ \hspace{-2mm} \sum_{j \in \mathcal{N}_{\tau_{k}}: t \in \mathcal{A}_j }  \hspace{-2mm}\bar{P}_j  +\hspace{-2mm} \sum_{j \in \mathcal{N}_{\tau_{k}}: t \notin \mathcal{A}_j} \hspace{-2mm} r(t) \bar{P}_j \\ \leq & \left ( \sum_{h \leq k} \sum_{j \in \mathcal{N}_{\tau_{h}}} \bar{P}_j \right )
	\end{aligned}
	\end{equation*}
	Hence, provided we can show $r(t)= \tilde{r} (t)$, we see that, $\tilde{K}_i \left ( x , d +  \sum_{j: t \notin \mathcal{A}_j} K_j (t,x,d) \right ) = \tilde{r}(t) \bar{P}_i = r(t) \bar{P}_i= K_i(t,x,d)$.

	To complete the proof of this case, notice that:
	\begin{equation*}
	\begin{aligned}
	r(t)& \left ( \sum_{j \in \mathcal{N}_{\tau_{k}}} \bar{P}_j \right ) \\
	& = r(t) \sum_{j \in \mathcal{N}_{\tau_{k}}: t \in \mathcal{A}_j} \hspace{-2mm}\bar{P}_j +\hspace{-2mm} \sum_{j \in \mathcal{N}_{\tau_{k}}: t \notin \mathcal{A}_j }\hspace{-2mm} r(t)\bar{P}_j \\
	& = r(t) \left (   \sum_{j \in \mathcal{N}_{\tau_{k}}: t \in \mathcal{A}_j} \hspace{-2mm} \bar{P}_j \right ) +\hspace{-2mm} \sum_{j \in
		\mathcal{N}_{\tau_{k}}: t \notin \mathcal{A}_j } \hspace{-2mm} K_j (t,x,d) \\
	& = d - \left ( \sum_{h < k} \sum_{j \in \mathcal{N}_{\tau_{h}}: t \in \mathcal{A}_j } \hspace{-2mm}\bar{P}_j \right ) 
	+ \hspace{-2mm}  \sum_{j \in
		\mathcal{N}_{\tau_{k}}: t \notin \mathcal{A}_j } \hspace{-2mm} K_j (t,x,d) \\
	& =  d - \left ( \sum_{h < k} \sum_{j \in \mathcal{N}_{\tau_{h}}} \bar{P}_j      \right ) + \sum_{j: t \notin \mathcal{A}_j} K_j (t,x,d) \\
	& = \tilde{r} (t) \sum_{j \in \mathcal{N}_{\tau_{k}} } \bar{P}_j
	\end{aligned}
	\end{equation*}
	\item Finally, when $i \in \mathcal{N}_{\tau_{k}}$, with  
	$d < \sum_{h < k} \sum_{j \in \mathcal{N}_{\tau_{h}}: t \in \mathcal{A}_j} \bar{P}_j$,
	we have $K_j (t,x,d)=0$ for all $j \in \mathcal{N}_{\tau_{h}}$ for all $h \geq k$. As a consequence:
	\begin{equation*}
	\begin{aligned}
    d + & \hspace{-2mm}\sum_{j : t \notin \mathcal{A}_j} K_j (t,x,d) =   d \hspace{-1mm} + \hspace{-1mm} \left ( \sum_{h<k}  \sum_{j \in \mathcal{N}_{\tau_{h}}: t \notin \mathcal{A}_j} \hspace{-2mm}K_j (t,x,d)  \right ) \\
	<& \left ( \sum_{h < k} \sum_{j \in \mathcal{N}_{\tau_{h}}: t \in \mathcal{A}_j} \hspace{-2mm} \bar{P}_j \right ) \hspace{-1mm} + \hspace{-1mm} \left (  \sum_{h<k}  \sum_{j \in \mathcal{N}_{\tau_{h}}: t \notin \mathcal{A}_j} \hspace{-2mm}K_j (t,x,d)  \right ) \\ \leq&  \left ( \sum_{h < k} \sum_{j \in \mathcal{N}_{\tau_{h}}: t \in \mathcal{A}_j} \hspace{-2mm}\bar{P}_j \right ) \hspace{-1mm} + \hspace{-1mm} \left ( \sum_{h<k}  \sum_{j \in \mathcal{N}_{\tau_{h}}: t \notin \mathcal{A}_j} \hspace{-2mm}\bar{P}_j \right )\\ 
	=&  \sum_{h < k} \sum_{j \in \mathcal{N}_{\tau_{h}}} \bar{P}_j.
	\end{aligned}
	\end{equation*}
	This proves that:
	$\tilde{K}_i \left ( x , d +  \sum_{j 
		: t \notin \mathcal{A}_j} K_j (t,x,d) \right ) = 0 = K_i(t,x,d)$.
\end{enumerate}
This concludes the proof of the Lemma.

\section{Proof of Theorem \ref{necandsufficient}}
We show first necessity of conditions (\ref{fundamentallimitation}). Let $d(\cdot)$ be a feasible demand signal. Then there exists policies $u_j (\cdot)$ such that:
\begin{enumerate}
	\item Total demand constraint: $\sum_j u_j(t)=d(t)$ for all $t \geq 0$;
	\item Power constraint: $0 \leq u_j (t) \leq \bar{P}_j$, for all $t \geq 0$;
	\item Availability constraints: $u_j(t) = 0$ for all $t \notin \mathcal{A}_j$;
	\item Energy constraint: $\int_0^{+\infty} u_j(t) \, dt \leq \bar{P}_j x_j(0)$.
\end{enumerate}
Hence, for every $\mathcal{W} \subset [0,+\infty)$ we see the following:
\[ \int_{\mathcal{W}} d(t) \, dt = \int_{\mathcal{W}} \sum_{j \in \mathcal{N} }
 u_j(t)\, dt =   \sum_{j \in \mathcal{N} } \int_{\mathcal{W}}
 	u_j(t)\, dt \]
 	\[ \qquad =  \sum_{j \in \mathcal{N} } \int_{\mathcal{W} \cap \mathcal{A}_j }
 	u_j(t)\, dt  \leq  \sum_{j \in \mathcal{N} } \min \{ \mu( \mathcal{W} \cap \mathcal{A}_j ), x_j(0) \} \bar{P}_j. \]  
 	This completes the necessity proof. 
 	
 	Conversely, let $d(\cdot)$ be unfeasible. Consider the associated map $\Lambda$, as
 	defined in equation (\ref{Lambdamap}) and (\ref{outsideenergy}). Let $\bar{\lambda}$ be a fixed point of the map (which always exists by Brower's fixed point Theorem) and $\hat{x}^{\bar{\lambda}} (0)$ the associated initial condition. Since $d$ is unfeasible, the set defined below is non-empty:
 	\[  \mathcal{N}_e = \{ j \in \mathcal{N}: \tilde{\varphi}_j (\sup \mathcal{T}, \hat{x}^{\bar{\lambda}}(0), \tilde{d}(\cdot )) <0 \}, \]
 	viz. of batteries which have negative energy at the end of the considered time-horizon, when the signal $\tilde{d}$ is dispatched.
 	Equivalently, by virtue of Lemma \ref{integratedlemma}, $\mathcal{N}_e$ can be expressed as
 		$\mathcal{N}_e = \{ j \in \mathcal{N}: \varphi_j (\sup \mathcal{T}, \hat{x}^{\bar{\lambda}}(0), d(\cdot )) <0 \}$.
 		
  Let $u_j(t)$ be defined as:
  \[ u_j(t) = \left \{ \begin{array}{rl} K ( t, \varphi(t,\hat{x}^{\bar{\lambda}} (0), d(\cdot), d(t)) & \textrm{if } t \in \mathcal{A}_j \\
  0 & \textrm{if } t \notin \mathcal{A}_j \end{array} \right .  \]
 The set $\mathcal{N}_e$ can equivalently be expressed as:
 $\mathcal{N}_e = \{  j \in \mathcal{N}: \int_{\mathcal{T}} u_j (t) dt > \bar{P}_j x_j(0) \}$,
 thanks to the fact that $\bar{\lambda}$ is a fixed point of $\Lambda$. 
The following implication is a consequence of the order-preserving property of the
maps $\tilde{\varphi}$ and $\varphi$ and of the definition of $K$:
\[   \exists \, i \in \mathcal{N}_e: u_i(t)>0 \; \Rightarrow u_j (t)= \bar{P}_j \; \forall \, j \notin \mathcal{N}_e: t \in \mathcal{A}_j  \]  	
Let $\mathcal{W}$ denote the set $\mathcal{W} = \bigcup_{j \in \mathcal{N}_e} \textrm{supp} (u_j)$.
Hence we might proceed to the following manipulations:
\begin{equation*}
    \begin{aligned}
\int_{\mathcal{W}}  d(t) \, dt =&
\int_{\mathcal{W}} \sum_{j \in \mathcal{N}} u_j(t) \, dt \\ =& \sum_{j \notin \mathcal{N}_e }
	\int_{\mathcal{W}} u_j (t) \, dt + \sum_{j \in \mathcal{N}_e } \int_{\mathcal{W}} u_j (t) \, dt \\
=& \sum_{j \notin \mathcal{N}_e }
\int_{\mathcal{W} \cap \mathcal{A}_j } u_j (t) \, dt + \sum_{j \in \mathcal{N}_e}
\int_{\mathcal{T}} u_j(t) \, dt \\
= & \sum_{j \notin \mathcal{N}_e }
\int_{\mathcal{W} \cap \mathcal{A}_j } \bar{P}_j \, dt +  \sum_{j \in \mathcal{N}_e}
\int_{\mathcal{T}} u_j(t) \, dt  \\
>&  \sum_{j \notin \mathcal{N}_e }  \mu( \mathcal{W} \cap \mathcal{A}_j) 
\bar{P}_j +  \sum_{j \in \mathcal{N}_e } \bar{P}_j x_j(0)\\ \geq&
\sum_{j \in \mathcal{N} } \min \{ \mu( \mathcal{W} \cap \mathcal{A}_j ),x_j(0)   \} \bar{P}_j.
\end{aligned}
\end{equation*}
This shows that (\ref{fundamentallimitation}) is violated and concludes the proof.

\section{Proof of Theorem \ref{timefailuremax}}
We first show by induction that $\tau^{*} \in [0, \tau_k]$ for all $k \in \mathbb{N}$. The claim is trivially true for $k=0$, given the initialization $\tau_0$.
Assume next that $\tau^{*} \in [0, \tau_k]$. We will show that $\tau^{*} \in
[0, \tau_{k+1}]$. Let $\bar{\lambda}_k$ be the fixed point of $\Lambda$ at the $k$-th iteration of the algorithm. Consider $\tilde{x}_k (0)$, the corresponding value of time-to-discharge with augmented energy proportional to $\bar{\lambda}_k$. We denote by $\tilde{x}_k(\cdot)$ the solution corresponding to a GGDDF policy for the fleet without availability constraints. It is known that this solution maximizes time-to-failure, viz.
\[  \begin{array}{rl} \mathcal{T}_f (\tilde{x}_k ) = 
 \max_{\tilde{u}(\cdot),\tilde{x}(\cdot)}& \mathcal{T}_f (\tilde{x}(\cdot)) \\
\textrm{s.t.}\quad \tilde{x}(0) &= \tilde{x}_k (0) \\
\dot{\tilde{x}}(t) &= - \Omega \tilde{u}(t), \quad \forall \, t \in [0,\tau_k] \\
0 \leq \tilde{u}_j(t) & \leq \bar{P}_j, \quad \forall \, t \in [0, \tau_k], \forall \, j \in \mathcal{N} \\
\tilde{d}_k(t) & =  \sum_{j \in \mathcal{N}} \tilde{u}_j(t). 
\end{array}
\]
Moreover, this also equals the maximum time-to-failure for demand signal $\tilde{d}$ over a restricted class of input policies, viz:
\[  \begin{array}{rl} \mathcal{T}_f (\tilde{x}_k )  =& 
 \max_{u(\cdot),\tilde{x}(\cdot)} \mathcal{T}_f (\tilde{x}(\cdot)) \\
\textrm{s.t.} \quad \tilde{x}(0) =& \tilde{x}_k (0) \\
\tilde{u}_j(t) = & \left \{ \begin{array}{rl} u_j(t) & t \in \mathcal{A}_j \\
K_j(t,\varphi(t,\tilde{x}_k(0),\tilde{d}_k),\tilde{d}_k(t)) & t \notin \mathcal{A}_j \end{array} \right . \\
\dot{\tilde{x}}(t) =& - \Omega \tilde{u}(t), \quad \forall \, t \in [0,\tau_k] \\
0 \leq& u_j(t) \leq \bar{P}_j, \quad \forall \, t \in [0, \tau_k], \forall \, j \in \mathcal{N} \\
\tilde{d}_k(t) = & \sum_{j \in \mathcal{N}} \tilde{u}_j(t). 
\end{array}
\]
Since the energy delivered by each device outside $\mathcal{A}_j$ in the above optimisation is equal to the additional energy provided at time $0$, we see that
$\tilde{x} (t) \geq x(t)$ where $x_j$ denotes the state evolution for input $u_j$ (which mathces $\tilde{u}_j$ in $\mathcal{A}_j$ but is zero otherwise).
Hence, the following inequality holds:
\begin{equation}
\label{maxi2}    
  \begin{array}{rl} \mathcal{T}_f (\tilde{x}_k ) & \geq 
 \max_{u(\cdot),x(\cdot)} \mathcal{T}_f (x(\cdot)) \\
\textrm{s.t.}\quad x(0) &= \Omega E(0) \\
\dot{x}(t) &= - \Omega u(t), \quad \forall \, t \in [0,\tau_k] \\
0 \leq u_j(t) & \leq \bar{P}_j, \quad \forall \, t \in [0, \tau_k], \forall \, j \in \mathcal{N} \\
u_j(t) & =0 \quad \forall \, t \notin \mathcal{A}_j \; \forall \; j \in \mathcal{N}\\
d(t) &=  \sum_{j \in \mathcal{N}} u_j(t). 
\end{array}
\end{equation}
The maximisation problem in (\ref{maxi2}), however, yields $\tau^*$ due to the induction hypotesis: $\tau^* \in [0, \tau_k]$. Hence,
$\tau_{k+1} = \mathcal{T}_f ( \tilde{x}_k(\cdot) ) \geq \tau^*$. Notice that $\tau_k$ is, by construction, a non-increasing and lower-bounded sequence. Hence it admits a limit $\bar{\tau}$. By the previous inequality, we see that 
$ \bar{\tau} = \lim_{k \rightarrow + \infty} \tau_k \geq \tau^*$. We need to show that equality holds. 

By definition of $\tilde{x}_k$, we have
$\tilde{x}_k (t) \geq 0 \; \forall \, t \in [0,\tau_{k+1}]$.
Let $k_n$ be any divergent sequence such that $\tilde{x}_{k_n}(0)$ converges to
some initial condition $\tilde{x}_{\infty}(0)$ as $n \rightarrow + \infty$.
Accordingly $\tilde{x}_{k_n}(\cdot)$ will converge to signal
$\varphi(t,\tilde{x}_{\infty} (0), d(\cdot))$ and $\tilde{d}_k(\cdot)$ will converge to some limit $\tilde{d}_{\infty} (t)$ fulfilling definition
$\tilde{d}_{\infty} (t) = d(t) + \sum_{j: t  \notin \mathcal{A}_j} K_j (t,\varphi(t, \tilde{x}_{\infty} (0), d(\cdot)), d(t))$. 
In particular, taking limits along subsequence $k_n$ yields $\tilde{x}_{\infty} (t) \geq 0$ for all $t \in [0, \bar{\tau}]$.
Hence $\tilde{d}_{\infty}$ is feasible over $[0, \bar{\tau}]$. Clearly $
\Lambda( \bar{\lambda}_{\infty} ) = \bar{\lambda}_{\infty}$ and therefore, by Theorem \ref{main}, $d$ is feasible for the fleet with partial availability constraints over the interval $[0,\bar{\tau}]$. Hence $\tau^{\star} \geq \bar{\tau}$. Since we already proved the opposite inequality, $\tau^* = \bar{\tau}$ which completes the proof.



\ifCLASSOPTIONcaptionsoff
  \newpage
\fi


\begin{thebibliography}{99}

    \bibitem{weitzel} T. Weitzel, and C. H. Glock, ``Energy management for stationary electric energy storage systems: A systematic literature review,'' \emph{European Journal of Operational Research}, Vol. 264, N.2, pp. 582–606, 2018.
    
    \bibitem{dynprog} R. Sioshansi, S. H. Madaeni, and P. Denholm, ``A dynamic programming approach to estimate the capacity value of energy storage,'' \emph{IEEE Transactions on Power Systems}, Vol. 29, N.1, pp. 395-403, 2014.
    
    \bibitem{appino} R.R. Appino, V. Hagenmeyer and T. Faulwasser, 2021. Towards optimality preserving aggregation for scheduling distributed energy resources. IEEE Transactions on Control of Network Systems, 8(3), pp.1477-1488.

	
	\bibitem{zhuzhang} D. Zhu, and Y. A. Zhang, ``Optimal coordinated control of multiple battery energy storage systems for primary frequency regulation,'' \emph{IEEE Transactions on Power Systems}, Vol. 34, N.1, pp. 555–565, 2019.

	\bibitem{ETA-pscc} M. {Evans}, S. H. {Tindemans} and D. {Angeli}, ``Robustly Maximal Utilisation of Energy-Constrained Distributed Resources,'' 2018 Power Systems Computation Conference (PSCC), pp.1-7, 2018.
	
	\bibitem{ETA-tops} M. P. {Evans}, S. H. {Tindemans} and D. {Angeli}, ``Minimizing Unserved Energy Using Heterogeneous Storage Units,'' IEEE Transactions on Power Systems, Vol. 34, N.5, pp. 3647-3656, 2019.
	
	\bibitem{EAST} M. P. Evans, D. Angeli, G. Strbac, and S. H. Tindemans, ``Chance-constrained ancillary service specification for heterogeneous storage
	devices,'' 2019 IEEE PES Innovative Smart Grid Technologies Europe (ISGT-Europe), pp. 1-5, 2019.
	
	\bibitem{zachary} J. Cruise and S. Zachary, ``Optimal scheduling of energy storage resources,'', arxiv, 2019.	
	
	\bibitem{ETA-tosg} M. P. {Evans}, S. H. {Tindemans} and D. {Angeli}, ``A Graphical Measure of Aggregate Flexibility for Energy-Constrained Distributed Resources,'' IEEE Transactions on Smart Grid, Vol. 11, N.1, pp. 106-117, 2020.
	
	\bibitem{nash} Nash P. and R. Weber, ``A simple optimizing model for reservoir control,'', Technical report, University of Cambridge, Cambridge, 1978.
	
	\bibitem{phtr} S. Zachary, S. Tindemans, M. Evans,J. Cruise and D. Angeli, ``Scheduling of Energy Storage,'' \emph{Philosophical Transactions of the Royal Society}, 2021.
	
	\bibitem{cdcbatteries} D. Angeli, Z. Dong and G. Strbac, ``Dispatch policies for heterogeneous storage devices with availability windows'', submitted to \emph{IEEE CDC 2022}.
	
	\bibitem{michaelphd} M. P. Evans, \emph{Characterising and Maximising Aggregate Flexibility of Heterogeneous Energy Storage Units}, PhD Thesis, Imperial College London, 2019.
	
	\bibitem{monotoneangeli} D. Angeli, and E. D. Sontag, ``Monotone control systems,'' \emph{IEEE Transactions on Automatic Control}, Vol. 48, N. 10, pp. 1684-1698, 2003, \texttt{doi: 10.1109/TAC.2003.817920}.
	
	\bibitem{translationinvariance} D. Angeli and E.D. Sontag, ``Translation-invariant monotone systems, and a global convergence result for enzymatic futile cycles,'' \emph{Nonlinear Analysis: Real World Applications}, Vol. 9, pp. 128 – 140, 2008.
	
	\bibitem{topology} Glen Bredon, \emph{Topology and Geometry}, \emph{Graduate Texts in Mathematics}, Vol. 139, Springer, 1993. 
	
	
	
	
	
	
	\bibitem{alinia2020online} Alinia, Bahram, Mohammad H. Hajiesmaili, Zachary J. Lee, Noel Crespi, and Enrique Mallada. ``Online EV scheduling algorithms for adaptive charging networks with global peak constraints.'' \emph{IEEE Transactions on Sustainable Computing} (2020).
	
	
	\bibitem{Ma2011decentralized} Ma, Zhongjing, Duncan S. Callaway, and Ian A. Hiskens. ``Decentralized charging control of large populations of plug-in electric vehicles.'' \emph{IEEE Transactions on control systems technology} 21, no. 1 (2011): 67-78.
	
	\bibitem{Xin2013a} Xin, Huanhai, Meidan Zhang, John Seuss, Zhen Wang, and Deqiang Gan. ``A real-time power allocation algorithm and its communication optimization for geographically dispersed energy storage systems.'' \emph{IEEE Transactions on Power Systems} 28, no. 4 (2013): 4732-4741.
	
	\bibitem{Chen2016state} Chen, Yue and Bu{\v{s}}i{\'c}, Ana and Meyn, Sean P. ``State estimation for the individual and the population in mean field control with application to demand dispatch.'' \emph{IEEE Transactions on Automatic Control} 62, no. 3 (2016): 1138-1149.
	
	\bibitem{Fortenbacher2015optimal} Fortenbacher, Philipp, G{\"o}ran Andersson, and Johanna L. Mathieu. ``Optimal real-time control of multiple battery sets for power system applications.'' In 2015 \emph{IEEE Eindhoven PowerTech}, pp. 1-6. IEEE, 2015.
	
	\bibitem{Liu2013planning} Liu, Jianzhe, Sen Li, Wei Zhang, Johanna L. Mathieu, and Giorgio Rizzoni. ``Planning and control of electric vehicles using dynamic energy capacity models.'' In 52nd \emph{IEEE Conference on Decision and Control}, pp. 379-384. IEEE, 2013.
	
	
	\bibitem{Jang2021large} Jang, Sunho, Necmiye Ozay, and Johanna Mathieu. ``Large-Scale Invariant Sets for Safe Coordination of Thermostatic Loads.'' (2021).
	
	
	\bibitem{Espinosa2020a} Espinosa, Luis A. Duffaut, and Mads Almassalkhi. ``A packetized energy management macromodel with quality of service guarantees for demand-side resources.'' \emph{IEEE Transactions on Power Systems} 35, no. 5 (2020): 3660-3670.
	
	\bibitem{Dall2017optimal} Dall’Anese, Emiliano, Swaroop S. Guggilam, Andrea Simonetto, Yu Christine Chen, and Sairaj V. Dhople. ``Optimal regulation of virtual power plants.'' \emph{IEEE Transactions on Power Systems} 33, no. 2 (2017): 1868-1881.
	
	
	\bibitem{Hao2014aggregate} Hao, He, Borhan M. Sanandaji, Kameshwar Poolla, and Tyrone L. Vincent. ``Aggregate flexibility of thermostatically controlled loads.'' \emph{IEEE Transactions on Power Systems} 30, no. 1 (2014): 189-198.
	
	
	\bibitem{Filippov} Filippov, Aleksei Fedorovich. Differential equations with discontinuous righthand sides: control systems. Vol. 18. \emph{Springer Science \& Business Media}, 2013.
	

\end{thebibliography}
\end{document}